\documentclass[12pt]{article}
\RequirePackage{amsthm,amsmath,amssymb,amscd}
\usepackage{verbatim}
\usepackage{color}
\usepackage{float} 
\usepackage[active]{srcltx}
\usepackage[usenames,dvipsnames]{xcolor}

\allowdisplaybreaks[4]

\theoremstyle{plain}
\newtheorem{theorem}{Theorem}[section]

\newtheorem{lemma}[theorem]{Lemma}
\newtheorem{property}[theorem]{Property}

\theoremstyle{definition}
\newtheorem{definition}[theorem]{Definition}

\newtheorem{remark}[theorem]{Remark}

\let\Section=\section
\def\section{\setcounter{equation}{0}\Section}
\newcommand{\R}{\mathbb{R}}
\newcommand{\bL}{\mathbb{L}}
\newcommand{\bC}{\mathbb{C}}

\newcommand{\ga}{\gamma\alpha}

\def\RR{\mathbb{R}}
\def\NN{\mathbb{N}}
\def\EE{\mathbb{E}}

\def\cF{{\cal F}}

\def\cH{{\cal H}}

\def\be{{\beta}}

\def\si{{\sigma}}

\def\al{{\alpha}}

\def\be{{\beta}}

\def\ga{{\gamma}}

\def\si{{\sigma}}

\def\La{{\Lambda}}

\def \eref#1{\hbox{(\ref{#1})}}

\usepackage{booktabs}

\newcommand{\Ceil}[1]{\left\lceil #1 \right\rceil}
\newcommand{\Indt}[1]{1_{\left\{#1 \right\}}}
\newcommand{\Norm}[1]{\left|\left|  #1   \right|\right|}

\usepackage{graphicx}
\usepackage{graphics}
\usepackage{epsfig}

\newcommand{\ud}{\ensuremath{\mathrm{d}}}
\newcommand{\FoxH}[5]{H_{#2}^{#1}\left(#3\:\middle\vert\: \begin{subarray}{l}#4\\[0.4em] #5\end{subarray}\right)}
\RequirePackage[colorlinks,citecolor=blue,urlcolor=blue]{hyperref}
\usepackage{enumerate}

\begin{document}

\title{Space-time fractional diffusions in Gaussian noisy environment\\
\medskip
\small dedicated to Professor Bernt \O ksendal on the occasion of his 70th birthday}

\author{  {\sc Le Chen}\thanks{L. Chen is partially
supported by   the Swiss National Foundation for Scientific
Research  fellowship (P2ELP2\_151796).} ,  \    {\sc Guannan Hu}, {\sc Yaozhong Hu}\thanks{Y. Hu is partially supported by a grant from the Simons Foundation \#209206. }\\
 Department of Mathematics \\
University of Kansas \\
Lawrence, Kansas, 66045 USA\\
\ and \
{\sc Jingyu Huang}\thanks{J. Huang is supported by the National Science Foundation under Grant No. DMS-1440140 while in residence at the Mathematical Research Institute in Berkeley, California, during the fall 2015 semester.  
} \\
Department of Mathematics\\
University of Utah\\
Salt Lake City, Utah, 84112 USA
}

\date{}
\maketitle

\begin{abstract}
This paper studies  the   linear stochastic partial differential equation   of fractional orders
both in time and space variables  $\left(\partial^\beta  + \frac{\nu}{2} (-\Delta)^{\alpha/2} \right) u(t,x)= \lambda u(t,x) \dot{W}(t,x)$,  where $\dot W$ is a general Gaussian noise
and $\be\in (1/2, 2)$,    {$\al\in (0, 2]$}.   The existence and uniqueness of the solution,  the moment
bounds of the solution are obtained by using the fundamental solutions of the corresponding
deterministic counterpart represented by the Fox H-functions. Along the way, we obtain some new properties of the fundamental solutions.
\end{abstract}


\section{Introduction}
In this article we consider the following linear stochastic partial differential equation of fractional orders
both in time and space variables:
\begin{align}  \label{E:SPDE}
\begin{cases}
\displaystyle \left(\partial^\beta  + \frac{\nu}{2} (-\Delta)^{\alpha/2} \right) u(t,x)= \lambda u(t,x) \dot{W}(t,x),&\qquad t>0,\: x\in\RR^d, \\[0.5em]
\displaystyle \left.\frac{\partial^k}{\partial t^k} u(t,x)\right|_{t=0}=u_k(x), &\qquad 0\le k\le \Ceil{\beta}-1, \:\: x\in\RR^d,
\end{cases}
\end{align}
with   {$\beta\in (1/2,2)$ and $\al\in (0, 2]$},
where $\Ceil{\beta}$ is the smallest integer greater than or equal to  $\beta$. Here and throughout the paper we denote
$\partial^k= \frac{\partial^k}{\partial t^k}$, $k\in\NN$.   We
limit our consideration to the above parameter ranges of $\beta$ and $\alpha$
since we plan to use
some particular properties of the corresponding {\it Fox H-functions} which will be proved
only for these parameter ranges. Now let us give more detailed explanation
on   the  terms appearing in the above equation.
The fractional derivative in time $\displaystyle \partial^\beta =\frac{\partial ^\beta}{\partial t^\beta}$
is understood in the {\it Caputo} sense:
\[
\partial^\beta  f(t) :=
\begin{cases}
\displaystyle
 \frac{1}{\Gamma(m-\beta)} \int_0^t  \frac{f^{(m)}(\tau)}{(t-\tau)^{\beta+1-m}}\ud
\tau\: &
\text{if $m-1<\beta<m$\;,}\\[1em]
\displaystyle
\frac{\ud^m}{\ud t^m}f(t)& \text{if $\beta=m$}\,,
\end{cases}
\]
where $t\ge 0$.
$\displaystyle \Delta=\sum_{i=1}^d \frac{\partial ^2}{\partial x_i^2}$ is the Laplacian with respect to spatial variables
and $ (-\Delta)^{\alpha/2}$ is the fractional Laplacian.
$\dot{W}$ is a  zero mean   Gaussian noise with the following  covariance structure
\[
\EE(\dot{W}(t,x)\dot{W}(s,y))=\gamma(t-s)\Lambda(x-y),
\]
where both (possibly generalized) functions $\gamma$ and $\Lambda$ are assumed to be nonnegative and nonnegative definite.
We denote by $\mu$ the Fourier transformation measure of  $\Lambda(x)$.
Namely,
\[
\La(x-y)=\frac{1} {(2\pi)^d}\int_{\RR^d} e^{i\xi(x-y)}\mu(\ud \xi)\,.
\]
{
This Fourier transform is understood in distributional sense (see Section 2).
}
When $\gamma(t)=\delta_0(t)$ and $\Lambda(x)=\delta_0(x)$,
this noise $\dot W$ reduces to the {\it space-time white noise}.
$\nu>0$ and $\lambda$ are some real valued parameters.
The given initial conditions $u_k(x)$ are assumed to be continuous and bounded   {functions}.
The product $u(t,x) \dot{W}(t,x)$ in equation \eqref{E:SPDE} is the Wick one (see e.g. \cite{huyan}).
So,  the equation will be understood in the Skorohod sense.  Let us   point out that some of our results can
also be extended to nonlinear equation (namely, replace $u(t,x) \dot{W}(t,x)$ in  \eqref{E:SPDE} by
$\si(u(t,x)) \dot{W}(t,x)$ for a   {Lipschitz} nonlinear function $\si$).
However, we limit ourselves to this linear case for two reasons: One is 
to simplify the presentation and to better explain the ideas and the other one is that we want to use the
chaos expansion method. 

The deterministic counterparts  of equation \eqref{E:SPDE}  have received a lot of attention and are called
anomalous diffusions.  They  appeared  in biological physics and other fields. Equation \eqref{E:SPDE} is an
anomalous diffusion in a Gaussian noisy environment.  More detailed
motivations for the study of this type of  equations  are given  in \cite{Chen14Time,ChenKimKim15,HuHu15, MijenaNana15ST}.
Here,  we   briefly mention   some recent results.

When  $\beta\in \{1,2\}$ is an integer, $\alpha=2$, the equation has been studied by many authors, see for example,   \cite{BC1,BC2, ChenDalang13Heat,ChenDalang14Wave,
HHNT}.
The work by Chen and Dalang $\cite{ChenDalang14FracHeat}$  deals with  the case
where $\beta=1$, $\alpha\in (1,2]$.
Khoshnevisan and Foondun \cite{FK08Int} and Song \cite{Song15} study a similar equation  with the $\alpha$-stable generator
$(-\Delta)^{\al/2}$ replaced by a general L\'evy generator.

When $\be\in (0, 1)$,  $\al=2$,  $\Delta$ is replaced by a general elliptic operator,
and $\dot W$ is a fractional noise, the equation was studied in \cite{HuHu15}.

When $\be\in (0, 1)$,  $\al=2$ and $\dot W$ is a fractional noise,  the smoothed equation
\[
\left(\partial^\beta  - \frac{\nu }{2}\Delta   \right) u(t,x)=
I_t^{1 - \beta} \left[ u(t,x) \dot{W}(t,x)\right]
\]
(see \eqref{E:SPDE-I} for a generalization) was studied   in \cite{Chen14Time}. 
In a series of papers \cite{FN15,MijenaNana14Int,MijenaNana15ST}, Nane and his coauthors
studied the case  $\alpha \in (0,2]$.

The case $\be\in (0, 1)$ corresponds to the slow diffusion  (subdiffusion).
For the fast diffusion case (super diffusion), i.e., $\beta\in (1,2)$, there have been  only a few  works.  %
The first author of this paper studies in \cite{Chen14Time} the smoothed equation with $\alpha=2$, $d=1$ and 
with space-time white noise.
The corresponding non-smoothed equation is studied recently in \cite{CHD15}.
Both papers \cite{Chen14Time,CHD15} deal with the nonlinear equation, i.e., $\rho(u)\dot{W}$ with $\rho$ being a
Lipschitz function. 

To study equation \eqref{E:SPDE}
the  important  tools that we shall use  are the   fundamental solutions corresponding to its deterministic counterpart.
Let us briefly   describe    them.
There are two fundamental solutions
\[
Z(t,x):=Z_{\alpha,\beta,d}(t,x)\quad\text{and}\quad
Y(t,x):=Y_{\alpha,\beta,d}(t,x)
\]
such that the solution $ u(t,x)$ to the following deterministic equation (the deterministic counterpart of
\eqref{E:SPDE}) 
\begin{align} \label{E:PDE}
\begin{cases}
\displaystyle \left(\partial^\beta  + \frac{\nu}{2} (-\Delta)^{\alpha/2} \right) u(t,x)= f(t,x),&\qquad t>0,\: x\in\RR^d, \\[0.5em]
\displaystyle \left.\frac{\partial^k}{\partial t^k} u(t,x)\right|_{t=0}=u_k(x), &\qquad 0\le k\le \Ceil{\beta}-1, \:\: x\in\RR^d,
\end{cases}
\end{align}
is represented by
\begin{align}\label{E:Duhamel}
 u(t,x) = J_0(t,x) +
\int_0^t \ud s \int_{\RR^d} \ud y\: f(s,y) Y(t-s,x-y),
\end{align}
where $f$ is a continuous and bounded function on $\RR_+\times\RR^d$, and 
$u_k$ is a continuous and bounded function on $\RR^d$.
In equation \eqref{E:Duhamel}  and throughout the paper, we denote
\begin{align}\label{E:J0}
J_0(t,x):=
\sum_{k=0}^{\Ceil{\beta}-1}\int_{\RR^d} u_{  {\Ceil{\beta}-1-k}}(y) \partial^{k} Z(t,x-y) \ud y\,.
\end{align}
Equations \eqref{E:Duhamel} and \eqref{E:J0}
say  that $Z$ and $Y$ are  the fundamental solutions corresponding respectively  to  the initial conditions and the inhomogeneous term of equation \eqref{E:PDE}.  For some parameter ranges of $\alpha$ and $\beta$, the fundamental solutions have been  studied in \cite{EK, Koc, Pskhu09, SchWyss89}. In  Section 3.1 we shall study them for all $\beta\in (1/2,2)$ and $\al\in (0, 2]$. In particular, we shall obtain 
some new properties such as the positivity of the fundamental solutions $Y$ and  $Z$.

Equation  \eqref{E:Duhamel}  motivates us to study the mild solution
to \eqref{E:SPDE} (see e.g. Definition \ref{D:Sol} below),  namely, the solution to the following
stochastic integral equation:
\begin{equation}\label{E:mild-1}
u(t,x)=J_0(t,x)+\int_0^t\int_{\RR^d}Y(t-s,x-y)u(s,y)  W(\ud s,\ud y).
\end{equation}
As in the classical case,  the above equation can be studied by using
  the It\^o-Wiener chaos expansion. To this end we need
to understand well the two fundamental solutions $Z$ and $Y$.
In particular, we need
their  nonnegativity and   some heat kernel like estimates.

The nonnegativity of some $Z$'s  is known. However, since
  $Y$ is the  {\it Riemann-Liouville} fractional derivative of $Z$,  
its nonnegativity is   a challenging problem.  There have been only few results:
  {As proved in Lemma 25 of \cite{Pskhu09}}, $Y_{2,\beta,d}$   {with $\beta\in (1,2)$} is nonnegative   if and only if $d\le 3$.
  {The one dimensional case is proved in \cite{MLP}, namely, $D_t Z_{\alpha,\beta,1}$, and hence $Y_{\alpha,\beta,1}$,  is  nonnegative
either if $1<\beta\le \alpha\le 2$, or if $\alpha\in(0,1]$ and $\beta \in (0,2)$}.
In this paper,  we shall show  the nonnegativity of $Y$ in the following three cases:
\begin{equation}\label{E:Cases}
\left\{\begin{alignedat}{3}
&\alpha\in (0,2],\quad &&\beta\in (1/2,1), \quad&& d\in\NN, \\
&\alpha\in (0,2],\quad &&\beta\in (1,\alpha), \quad && d=1,\\
&\alpha=2, \quad &&\beta\in (1,2), \quad && d=2,3.
\end{alignedat}
\right.
\end{equation}
This includes the above mentioned results as special cases. Let us also point out that
for the smoothed SPDE, only the fundamental solution $Z$ is needed, which is usually more
regular  than the fundamental solution $Y$.

When $\be=1$ and $\al=2$, to show the solution of \eqref{E:SPDE} is square integrable, it is assumed in
\cite{HHNT} and \cite{HHLNT}
that the covariance of noise satisfies the following   {conditions:}
\begin{description}
\item{(i)}\ $\ga$ is locally integrable;
\item{(ii)}\
 {\it Dalang's condition} $\displaystyle \int _{\RR^d} \frac{\mu(\ud\xi)}{1+ |\xi|^{2 }} < \infty $ is satisfied
  (see also \cite{Dalang99Extending,FK08Int}).
\end{description}
For the existence and uniqueness of the solution to the general equation \eqref{E:SPDE},   Dalang's condition
will be replaced by the following condition:
\begin{equation}\label{E:DLcond}
\int _{\RR^d} \frac{\mu(\ud\xi)}{1+ |\xi|^{2\alpha-\alpha/\beta}} < \infty\,.
\end{equation}
It is obvious that if it is formally set   $\al=2$ and $\be=1$,  then 
\eqref{E:DLcond} is reduced to the usual Dalang's condition.

\bigskip
The remaining part of the  paper is organized as follows.
We first specify the noise structure and present  the definition of the solution   in Section \ref{Sec:Pre}.
The main results  are Theorem \ref{T:ExUni} on the existence and uniqueness of the mild
solution
and Theorem \ref{T:Rietz} on the moment bounds of the solution stated in Section \ref{Sec:Main}.
The proof of these  two theorems are based on some properties of the fundamental solutions
represented in terms of the Fox H-functions. These results
themselves are of particular interest and importance.
We also list them as Theorem \ref{T:PDE} and  Theorem \ref{T:Stable} in Section \ref{Sec:Main}.
The properties of the fundamental solutions (Theorem \ref{T:PDE}) are  proved in Section \ref{Sec:PDE}
by using the Fox H-functions.
In Section \ref{Sec:Stable}, we obtain  an expression of the density function for the $d$-dimensional spherically symmetric $\alpha$-stable distribution
 - an auxiliary result (Theorem \ref{T:Stable}) which is used in the proof of Theorem \ref{T:PDE}.
The existence and uniqueness result (Theorem \ref{T:ExUni}) of the solution to
\eqref{E:SPDE} is proved in Section \ref{Sec:ExUni}.
In Section \ref{Sec:Rietz}, we prove the explicit moment bounds when $\Lambda$ is the Riesz kernel.

Our main results (Theorem \ref{T:ExUni})
 assume that the fundamental solutions are nonnegative. However, when  $1<\beta <2$ and
when the dimension is high, the nonnegativity of the fundamental solution $Y$ is not known yet. In this case,
we  shall show   {in Theorem \ref{beta.2}} the existence and uniqueness of the solution of \eqref{E:SPDE} for some specific
Gaussian noise whose covariance function
 $\Lambda$ is the Riesz kernel.
Finally, in Appendix \ref{Sec:H} we collect some knowledge  on the Fox H-function which we need in this paper.

\section{Preliminary}\label{Sec:Pre}
Let us start by introducing some basic notions on Fourier transforms.
The space of real-valued infinitely differentiable
functions on $\RR^d$ with compact support is denoted by $\mathcal{D} (
\mathbb{R}^d)$ or $\mathcal{D}$. The space of Schwartz functions is
denoted by $\mathcal{S} ( \mathbb{R}^d)$ or $\mathcal{S}$. Its dual,
the space of tempered distributions, is denoted by  $\mathcal{S}' (
\mathbb{R}^d)$ or $\mathcal{S}'$. The Fourier
transform is defined with the normalization
\[ \mathcal{F}u ( \xi)  = \int_{\mathbb{R}^d} e^{- i
   \xi \cdot  x} u ( x) \ud x, \]
so that the inverse Fourier transform is given by $\mathcal{F}^{- 1} u ( \xi)
= ( 2 \pi)^{- d} \mathcal{F}u ( - \xi)$.

\smallskip
Similarly to \cite{HHNT}, on a complete probability space
$(\Omega,\mathcal{F},P)$ we consider a Gaussian noise $W$ encoded by a
centered Gaussian family $\{W(\varphi) ; \, \varphi\in
\mathcal{D}(\R_+\times \R^{d})\}$, whose covariance structure
is given by
\begin{equation}\label{cov1}
\EE \left (  W(\varphi) \, W(\psi) \right)
= \int_{\R_{+}^{2}\times\R^{2d}}
\varphi(s,x)\psi(t,y)\gamma(s-t)\Lambda(x-y)\ud x\ud y\ud s\ud t,
\end{equation}
where $\gamma: \R \rightarrow \R_+$ and $\Lambda: \R^d \rightarrow
\R_+$ are  nonnegative definite
functions and the Fourier transform $\mathcal{F}\Lambda=\mu$ such that $\mu(\ud\xi)
$ is a tempered
measure, that is, there is an integer $m \geq 1$ such that
$\int_{\R^d}(1+|\xi|^2)^{-m}\mu(\ud \xi)< \infty$.
Throughout the paper, we assume that $\ga$ is locally integrable and we denote
\begin{align}\label{E:Ct}
 C_t := 2 \int_0^t \gamma(s)\ud s, \quad t>0.
\end{align}

\smallskip

 Let $\mathcal{H}$  be the completion of
$\mathcal{D}(\R_+\times\R^d)$
endowed with the inner product
\begin{eqnarray}\label{innprod1}
\langle \varphi , \psi \rangle_{\mathcal{H}}&=&
\int_{\R_{+}^{2}\times\R^{2d}}
\varphi(s,x)\psi(t,y)\gamma(s-t)\Lambda(x-y) \, \ud x\ud y\ud s\ud t\\ \notag
&=&\frac{1}{(2\pi)^d} \int_{\R_{+}^{2}\times\R^{d}}  \mathcal{F} \varphi(s,\xi) \overline{ \mathcal{F} \psi(t,\xi)}\gamma(s-t) \mu(\ud\xi) \, \ud s\ud t,
\end{eqnarray}
where $\mathcal{F} \varphi$ refers to the Fourier transform with respect to the space variable only.
 The mapping $\varphi \rightarrow W(\varphi)$ defined on $\mathcal{D}(\R_+\times\R^d)$  can be extended
  to a linear isometry between
$\mathcal{H}$  and the Gaussian space
spanned by $W$.  We will denote this isometry by
\begin{equation*}
W(\phi)=\int_0^{\infty}\int_{\R^d}\phi(t,x)W(\ud t,\ud x),\quad\text{for $\phi \in \mathcal{H}$.}
\end{equation*}
Notice that if $\phi$ and $\psi$ are in
$\mathcal{H}$, then
$\EE \left( W(\phi)W(\psi)\right) =\langle\phi,\psi\rangle_{\mathcal{H}}$.

We will denote by $D$ the derivative operator in the sense of
Malliavin calculus. That is, if $F$ is a smooth and cylindrical
random variable of the form
\begin{equation*}
F=f(W(\phi_1),\dots,W(\phi_n))\,,
\end{equation*}
with $\phi_i \in \mathcal{H}$, $f \in C^{\infty}_p (\R^n)$ (namely $f$ and all
its partial derivatives have polynomial growth), then $DF$ is the
$\mathcal{H}$-valued random variable defined by
\begin{equation*}
DF=\sum_{j=1}^n\frac{\partial f}{\partial
x_j}(W(\phi_1),\dots,W(\phi_n))\phi_j\,.
\end{equation*}
The operator $D$ is closable from $L^2(\Omega)$ into $L^2(\Omega;
\mathcal{H})$  and we define the Sobolev space $\mathbb{D}^{1,2}$ as
the closure of the space of smooth and cylindrical random variables
under the norm
\[
\|F\|_{1,2}=\sqrt{\EE[F^2]+\EE[\|DF\|^2_{\mathcal{H}}]}\,.
\]
We denote by $\delta$ the adjoint of the derivative operator given
by the duality formula
\begin{equation}\label{E:dual}
\EE \left( \delta (u)F \right) =\EE \left( \langle DF,u
\rangle_{\mathcal{H}}\right) ,
\end{equation}
for all $F \in \mathbb{D}^{1,2}$ and any element $u \in L^2(\Omega;
\mathcal{H})$ in the domain of $\delta$. The operator $\delta$ is
also called the {\it Skorohod integral} because in the case of the
Brownian motion, it coincides with an extension of the It\^o
integral introduced by Skorohod. We refer to Nualart \cite{Nua}
for a detailed account of the Malliavin calculus with respect to a
Gaussian process.

With the Skorohod integral introduced,
the definition of the solution to equation \eref{E:SPDE} can be stated as follows.

\begin{definition}\label{D:Sol} Let $Z$ and $Y$ be the fundamental solutions defined by
\eqref{E:PDE} and \eqref{E:Duhamel}.
An adapted random field $ \{ u={u(t,x): \:t\geq 0, x\in \mathbb{R}^d} \} $ such that
$\EE \left[u^2(t,x)\right]<+\infty$ for all $(t,x)$
is a {\it mild solution} to \eqref{E:SPDE}, if for all $(t,x)\in\R_+\times \mathbb{R}^d$,
the process
\[
\left\{Y(t-s, x-y)u({s,y})1_{[0,t]}(s): \: s\ge0,\: y\in \mathbb{R}^d \right\}
\]
 is Skorohod integrable (see \eqref{E:dual}), and
 $u$ satisfies
 \begin{equation}\label{E:mild}
u(t,x)=J_0(t,x)+\int_0^t\int_{\RR^d}Y(t-s,x-y)u(s,y)  W(\ud s,\ud y)
\end{equation}
 almost surely for all $(t,x)\in\RR_+\times\RR^d$,  where $J_0(t,x)$ is defined by
 \eqref{E:J0}.
%
\end{definition}

The main ingredient in proving the existence and uniqueness of the solution is the Wiener chaos expansion, to which we now turn.

Suppose that $u=\{u({t,x}); t\geq 0, x \in \RR^d\}$ is a square integrable solution to equation \eref{E:mild}. Then for all fixed $(t,x)$ the random variable $u({t,x})$ admits the following Wiener chaos expansion
\begin{equation}
u({t,x})=\sum_{n=0}^{\infty}I_n(f_n(\cdot,\cdot,t,x))\,,
\end{equation}
where for each $(t,x)$, $f_n(\cdot,\cdot,t,x)$ is a symmetric element in
$\mathcal{H}^{\otimes n}$.
Then, as in \cite{Hu, Hu1, HuHu15}, to show the existence and uniqueness of the solution it suffices to show that for all $(t,x)$ we have
\begin{equation}\label{eq: L2 chaos}
\sum_{n=0}^{\infty}n!\|f_n(\cdot,\cdot,t,x)\|^2_{\mathcal{H}^{\otimes n}}< \infty\,.
\end{equation}

For technical reasons (see \eref{Davar id} below), we will assume, throughout the paper,  the following properties on  $\Lambda$:
\begin{itemize}
\item $\Lambda: \RR^d \rightarrow [0, \infty]$ is a continuous function, where $[0, \infty]$ is the usual one-point compactification of $[0, \infty)$.
\item $\Lambda(x) < \infty$ if and only if $x \neq 0$ or $\mathcal{F}(\Lambda) \in L^{\infty}(\RR^d)$ and $\Lambda(x) < \infty$ when $x \neq 0$.
\end{itemize}
With these two assumptions, according to Lemma 5.6 in \cite{KX}, for any  Borel probability measures $\nu_1(dx)$ and $\nu_2(dx)$,  the following identity holds,
\begin{equation}\label{Davar id}
\int_{\RR^d} \int_{\RR^d} \Lambda(x-y) \nu_1(\ud x) \nu_2(\ud y) =
\frac{1}{(2\pi)^d} \int_{\RR^d} \mathcal{F} \nu_1(\xi) \overline{\mathcal{F}\nu_2(\xi)} \mu(\ud \xi)\,.
\end{equation}
In particular, the above result can be applied to the case
when $\nu_1(\ud x)=f_1(x)\ud x$ and $\nu_2(\ud x)= f_2(x)\ud x$
for two nonnegative functions $f_1$ and $f_2\in L^1 (\RR^d)$.

\section{Main results} \label{Sec:Main}

\subsection{Fundamental solutions: formulas and nonnegativity}
Our first result is concerned with  the fundamental solutions to \eqref{E:PDE} stated in the following theorem.
We need the   two parameter  {\em Mittag-Leffler function} $E_{\alpha,\beta}(z)$:
\begin{equation}
E_{\alpha,\beta}(z):=\sum_{n=0}^{\infty}\frac{z^n}{\Gamma(\alpha n + \beta)}\,, \quad
\Re(\alpha)>0,\: \beta\in\mathbb{C},\: z\in\mathbb{C}\,,
\end{equation}
where $\Re(\alpha)$ is the real part of the complex number $\alpha$.
When $\beta = 1$, we also write $E_{\alpha}(z):= E_{\alpha,1}(z)$.
The $H$-functions
appearing  in the following  theorem and their properties are   given in the appendix.
\begin{theorem}\label{T:PDE}
The  fundamental solutions to \eqref{E:PDE} are given by
\begin{align}\label{E:Zab}
 Z(t,x):=Z_{\alpha,\beta,d}(t,x)= \pi^{-d/2} t^{\Ceil{\beta}-1} |x|^{-d}
 \FoxH{2,1}{2,3}{\frac{ |x|^\alpha}{2^{\alpha-1}\nu t^\beta}}{(1,1),\:(\Ceil{\beta},\beta)}
 {(d/2,\alpha/2),\:(1,1),\:(1,\alpha/2)}
\end{align}
and
\begin{align}\label{E:Yab}
 Y(t,x):=Y_{\alpha,\beta,d}(t,x)= \pi^{-d/2} |x|^{-d}t^{\beta-1}
 \FoxH{2,1}{2,3}{\frac{ |x|^\alpha}{2^{\alpha-1}\nu t^\beta}}
 {(1,1),\:(\beta,\beta)}{(d/2,\alpha/2),\:(1,1),\:(1,\alpha/2)}\,.
\end{align}
If $\beta\in(1,2)$,  then
\begin{align}\label{E:Z*ab}
 Z^*(t,x):=Z_{\alpha,\beta,d}^*(t,x) =
\frac{\ud}{\ud t}  Z_{\alpha,\beta,d}(t,x) =
 \pi^{-d/2} |x|^{-d}
 \FoxH{2,1}{2,3}{\frac{ |x|^\alpha}{2^{\alpha-1}\nu t^\beta}}{(1,1),\:(1,\beta)}
 {(d/2,\alpha/2),\:(1,1),\:(1,\alpha/2)}.
\end{align}
The   Fourier transforms of the fundamental solutions are given by the following:
 \begin{align}
 \label{E:FZ}
\cF Z(t,\cdot)(\xi) &= t^{\Ceil{\beta}-1} E_{\beta,\Ceil{\beta}}(-2^{-1}\nu t^\beta |\xi|^\alpha),\\
 \cF Y(t,\cdot)(\xi)& = t^{\beta-1} E_{\beta,\beta}(-2^{-1}\nu t^\beta |\xi|^\alpha),
 \label{E:FY}\\
 \cF Z^*(t,\cdot)(\xi) &= E_{\beta}(-2^{-1}\nu t^\beta |\xi|^\alpha), \quad \text{if $\beta\in (1,2)$};
  \label{E:FZ*}
 \end{align}
Moreover,  we have the following results on the positivity of the fundamental solutions.
 \begin{enumerate}[(a)]
  \item If $\beta\in (0,1]$,   {$d\in\NN$ and $\alpha\in(0,2]$}, then both   $Z(t,x)$ and $Y(t,x)$ are nonnegative;
  \item If  $\beta\in (1,2)$, $d\in \{2,3\}$, and $\alpha=2$, then both   $Z(t,x)$ and $Y(t,x)$ are nonnegative;
  \item If  $\beta\in (1,2)$, $d=1$ and $\alpha\in [\beta, 2]$, then all   $Z(t,x)$, $Y(t,x)$ and $Z^*(t,x)$ are nonnegative.
\end{enumerate}
\end{theorem}

The proof of this theorem is given in Section  \ref{Sec:PDE}.

\begin{remark}
Here are some known special cases:
\begin{enumerate}[(1)]
 \item When $\alpha=2$ and $\beta\in (0,1)$,
 it is proved in \cite{Koc,SchWyss89} and in \cite{EK}, respectively,  that
\begin{align}\label{E:Z2}
 Z_0(t,x)= \pi^{-d/2} |x|^{-d} \FoxH{2,0}{1,2}{\frac{ |x|^2}{2\nu t^\beta}}{(1,\beta)}{(d/2,1),(1,1)},
\end{align}
and
\begin{align}\label{E:Y2}
 Y_0(t,x)=\pi^{-d/2} |x|^{-d} t^{\beta-1}\FoxH{2,0}{1,2}{\frac{ |x|^2}{2\nu t^\beta}}{(\beta,\beta)}{(d/2,1),(1,1)},
\end{align}
which correspond to our
$Z_{2,\beta,d}(t,x)$ and $Y_{2,\beta,d}(t,x)$, respectively.
The equivalence is clear by applying Property \ref{Prop:Reduction}.
For $Z_{2,\beta,d}$, see also \cite[Chapter 6]{KilbasEtc06}.
\item When $\alpha=2$ and $\beta\in (0,2)$, it is proved in \cite{Pskhu09} that
\begin{align}\label{E:YW2}
 \Gamma_{\beta,d}(t,x)=\pi^{-d/2} |x|^{-d} t^{\beta-1}
 \FoxH{2,0}{1,2}{\frac{ |x|^2}{4 t^\beta}}{(\beta,\beta)}{(d/2,1),(1,1)},
\end{align}
which corresponds to our $Y_{2,\beta,d}$ with $\nu=2$.
\item In \cite{MLP}, the fundamental solution $Z^*_{\alpha,\beta,d}(t,x)$ has been studied for all $\alpha, \beta\in(0,2)$ and $d=1$.
From the Mellin-Barnes integral representation (6.6) of \cite{MLP}, we  see that the reduced Green function of \cite{MLP}
can be expressed by using the Fox H-function:
\begin{align}\label{E:Kabt}
K_{\alpha,\beta}^\theta(x)=\frac{1}{|x|}
\FoxH{2,1}{3,3}{|x|^\alpha}
{(1,1),\:(1,\beta),\:(1,\frac{\alpha-\theta}{2})}
{(1,1),\:(1,\alpha),\:(1,\frac{\alpha-\theta}{2})}, \quad x\in\RR,
\end{align}
where $\alpha$ and $\beta$ have the same meaning as in this paper and $\theta$ is the skewness: $|\theta|\le \min(\alpha,2-\alpha)$.
For the symmetric $\alpha$-stable case, i.e., $\theta=0$, this expression can be simplified by using
the definition of the Fox H-function and the fact that (see, e.g., \cite[5.5.5]{NIST2010})
\begin{align}\label{E:GGPI}
 \frac{\Gamma(1+\alpha s)}{\Gamma(1+\alpha s/2)} = \frac{1}{\sqrt{\pi}}2^{\alpha s} \Gamma(1/2 + \alpha s/2).
\end{align}
Hence,
\begin{align}
K_{\alpha,\beta}^0(x)=\frac{1}{\sqrt{\pi}|x|}
\FoxH{2,1}{2,3}{(|x|/2)^\alpha}
{(1,1),\:(1,\beta)}
{(1/2,\alpha/2),\:(1,1),\:(1,\alpha/2)},\quad x\in\RR.
\end{align}
This implies that the  fundamental solution in  \cite[(1.3)]{MLP}
\[
G^0_{\alpha,\beta}(x,t)=t^{-\beta/\alpha}K^0_{\alpha,\beta}(t^{-\beta/\alpha} x) =
\frac{1}{\sqrt{\pi}|x|}
\FoxH{2,1}{2,3}{\frac{|x|^\alpha}{2^\alpha t^\beta}}
{(1,1),\:(1,\beta)}
{(1/2,\alpha/2),\:(1,1),\:(1,\alpha/2)}
\]
corresponds to our $Z^*_{\alpha,\beta,1}(t,x)$ with $\nu=2$.
\end{enumerate}
\end{remark}

The proof of the nonnegativity part in  Theorem \ref{T:PDE} requires a representation of the
spherically symmetric $\alpha$-stable distribution from the Fox H-function,
which is of interest by itself.
The one-dimensional case can be found  in \cite{MLP};  see Remark \ref{R:Stable} below.

\begin{theorem}\label{T:Stable}
Let $X$ be a centered, $d$-dimensional spherically symmetric $\alpha$-stable random variable with $\alpha\in (0,2]$.
Then the characteristic function and the density of $X$ are, respectively,
\begin{align}
\label{E_:Stable}
f_{\alpha,d}(\xi)= \exp\left(-|\xi|^\alpha\right),\quad\xi\in\RR^d,
\end{align}
and
\begin{align}
\label{E:Stable}
\rho_{\alpha,d}(x)= \pi^{-d/2} |x|^{-d}
\FoxH{1,1}{1,2}{(|x|/2)^\alpha}{(1,1)}{(d/2,\alpha/2),\:(1,\alpha/2)},\quad x\in\RR^d.
\end{align}
\end{theorem}

The proof of this theorem is given in Section  \ref{Sec:Stable}.

\begin{remark}\label{R:Stable}
When $d=1$, the formula \eqref{E:Stable} yields  a result  in \cite{MLP}.
In particular, as proved in \cite{MLP} (see \eqref{E:Kabt}), when $d=1$,  we have
\[
\rho_{\alpha,1}(x)= |x|^{-1}
\FoxH{1,1}{2,2}{|x|^\alpha}{(1,1),\:(1,\alpha/2)}{(1,\alpha),\:(1,\alpha/2)}=
\pi^{-1/2} |x|^{-1}
\FoxH{1,1}{1,2}{(|x|/2)^\alpha}{(1,1)}{(1/2,\alpha/2),\:(1,\alpha/2)},
\]
where the second equality is due to \eqref{E:GGPI} and the definition of the Fox H-function.
\end{remark}

\subsection{Existence and uniqueness of solutions to the SPDE}
The following is one of the main theorems of the paper. 
\begin{theorem}\label{T:ExUni}
Assume the following conditions:
\begin{enumerate}[(1)]
\item $Y_{\alpha,\beta,d}(t,x)$ is nonnegative;
\item $\beta \in (1/2,2)$ and $\alpha\in (0,2]$;
\item $\gamma$ is locally integrable;
\item $\mu$ satisfies Dalang's  condition \eqref{E:DLcond};
\item The initial conditions are such that for all $t> 0$,
 \begin{align}\label{E:HatCt}
\widehat{C}_t:=\sup_{y\in\RR^d,\: s\in[0,t]} |J_0(s,y)|<+\infty.
\end{align}
\end{enumerate}
Then relation \eref{eq: L2 chaos} holds for each $(t,x)$. Consequently, equation \eref{E:SPDE} admits a unique mild solution in the sense of Definition \ref{D:Sol}.
\end{theorem}
The proof of this theorem is given in Section  \ref{Sec:ExUni}.

\begin{remark}
From Theorem \ref{T:PDE}, it follows that  the three cases in \eqref{E:Cases} satisfy the above assumptions   {(1) and (2)}.
Moreover,
if $u_k\in L^{\infty}(\R^d)$ for the first two cases in \eqref{E:Cases},
or if $u_0\in L^{\infty}(\R^d)$ and $u_1(x)\equiv u_1$ is a constant for
the last case in \eqref{E:Cases},
then by Lemma \ref{L:J0} below,
\[
|J_0(t,x)|
\le
\Norm{u_0}_{L^{\infty}(\RR^d)}+t^{\beta-1}\Norm{u_1}_{L^{\infty}(\RR^d)}\Indt{\beta>1}.
\]
Hence the assumption (5) is also satisfied.
The Dalang condition \eqref{E:DLcond} imposes a further restriction on the possible values of $(\alpha,\beta)$
due to the spatial correlation $\Lambda(x)$.
\end{remark}

\begin{remark}[Space-time white noise case]
When the noise $\dot{W}$ is a space-time white noise, i.e., $\gamma(t)=\delta_0(t)$ and $\Lambda(x)=\delta_0(x)$,
then Dalang's condition \eqref{E:DLcond} becomes
\begin{align}\label{E:WN-d}
\frac{d}{\alpha} + \frac{1}{\beta}<2.
\end{align}
This condition implies that $\beta>1/2$.
In particular, if $\alpha=2$ and $d=1$, then \eqref{E:WN-d} reduces to
\[
\beta>2/3\:,
\]
which recovers the condition in \cite{CHD15} and \cite[Section 5.2]{ChenKimKim15}.
If $\beta=1$ and $d=1$,   then this condition becomes
\begin{align}\label{E:a>1}
\alpha>1\:,
\end{align}
which recovers the condition in \cite{ChenDalang14FracHeat}.
\end{remark}


\subsection{The smoothed equation}

The methodology used in the proof of Theorem \ref{T:ExUni} can also be used to study  the following  equation
\begin{equation}\label{E:SPDE-I}
\left(\partial^\beta  + \frac{\nu}{2} (-\Delta)^{\frac{\alpha}{2}}\right) u(t,x)=
I_t^{\lceil \beta \rceil - \beta} \left[ u(t,x) \dot{W}(t,x)\right]\,,
\end{equation}
with the same initial conditions as \eqref{E:SPDE}.
Here $I_t^{\beta}$ is the   {\it Riemann-Liouville fractional integral}  of order $\beta$
(with an  abuse of the notation $\beta$):
\begin{equation*}
I_t^{\beta}f(t)=\frac{1}{\Gamma(\beta)} \int_0^t (t-s)^{\beta-1} f(s) \ud s, \quad \text{for $t>0$ and $\beta>0$}\,.
\end{equation*}
Due to the fractional integral in equation \eref{E:SPDE-I} which plays a
smoothing role, the mild formulation for the solution can be expressed
by using $Z(t,x)$ only, namely,
\begin{equation}\label{E:mild-I}
u(t,x)=J_0(t,x)+\int_0^t\int_{\RR^d}Z(t-s, x-y)u(s,y)
  {W(\ud s,\ud y)}\,.
\end{equation}
Then, using the same procedure as in the proof of Theorem \ref{T:ExUni}, we have the following result.
\begin{theorem}\label{T:ExUni-I}
Assume the conditions (3) and (5) in Theorem \ref{T:ExUni} and the other
conditions are replaced by the following:
\begin{enumerate}
 \item[(1')] $Z_{\alpha,\beta,d}(t,x)$ is nonnegative;
 \item[(2')]  $\beta\in (1/2,1]\cup (3/2,2)$ and $\alpha\in (0,2]$;
 \item[(4')] $\mu$ satisfies
\begin{equation}\label{E:DLcon-I}
\int _{\RR^d} \frac{\mu(\ud\xi)}{1+ |\xi|^{\alpha(2\Ceil{\beta}-1)/\beta}} < \infty\,.
\end{equation}
\end{enumerate}
Then relation  \eqref{eq: L2 chaos} holds for each $(t,x)$.
Consequently, the smoothed equation \eref{E:SPDE-I} admits a unique mild solution in the sense of Definition \ref{D:Sol} with
$Y$   replaced by $Z$.
\end{theorem}
\begin{remark} The condition (1') is different and  is usually easier to verify  than the condition (1) in Theorem \ref{T:ExUni}.  When $\beta\in  (1/2, 1]$, the condition \eqref{E:DLcon-I} becomes
$\int _{\RR^d} \frac{\mu(\ud\xi)}{1+ |\xi|^{ \alpha /\beta}} < \infty$ which is also
 weaker than \eqref{E:DLcond} (since $\be\le 1$).
When $\beta\in  (3/2,2)$, the condition \eqref{E:DLcon-I} becomes
$\int _{\RR^d} \frac{\mu(\ud\xi)}{1+ |\xi|^{3\alpha /\beta}} < \infty$ which is also
 weaker than \eqref{E:DLcond} (since $\be<2$).
\end{remark}

The proof is essentially the same as that for Theorem \ref{T:ExUni}, the only change in the proof worthy to be pointed out is that instead of computing the integral
\begin{equation*}
\int_0^{\infty} w^{2(\Ceil{\beta} -1) } E_{\beta,\Ceil{\beta}}^2 (-2^{-1}\nu w^{\beta} |\xi|^\alpha) \ud w\,,
\end{equation*}
we now need to compute the integral
\begin{align}\notag
\int_0^{\infty} w^{2(\Ceil{\beta}-1)}E_{\beta,\Ceil{\beta}}^2(-2^{-1}\nu w^{\beta}|\xi|^\alpha) \ud w
&= \frac{(2/\nu)^{(2\Ceil{\beta}-1)/\beta}}{|\xi|^{\alpha(2\Ceil{\beta}-1)/\beta}}
\int_0^{\infty} \frac{1}{\beta} s^{\frac{1}{\beta}(2\Ceil{\beta}-\beta-1)} E_{\beta,\Ceil{\beta}}^2(-\nu s)\ud s\\
&= \frac{C}{|\xi|^{\alpha(2\Ceil{\beta}-1)/\beta}}\,.
\label{E:CxiI}
\end{align}
The integrability condition of the above equation at zero and at infinity
implies that $\beta>0$ and $\beta\in (1/2,1]\cup (3/2,2]$ (which is equivalent
to $\Ceil{\beta}<\beta+1/2$), respectively.
  {Note that this} condition on $\beta$ is more restrictive
than the condition $\beta\in (0,2)$ in \cite{Chen14Time}.

\begin{remark}[Space-time white noise case]
When the noise $\dot{W}$ is a space-time white noise(namely  $\mu(d\xi)=d\xi$),
then Dalang's condition \eqref{E:DLcon-I} becomes
\begin{align}
\label{E':WN-d}
d<\alpha(2\Ceil{\beta}-1)/\beta\qquad\text{or}\qquad \frac{d}{\alpha} +\frac{1}{\beta} < \frac{2\Ceil{\beta}}{\beta}.
\end{align}
In particular, if $\alpha=2$ and $d=1$, then this condition reduces to $\beta<2$.
If $\beta=1$ and $d=1$, then this condition becomes \eqref{E:a>1},
which recovers the condition in \cite{ChenDalang14FracHeat}.
\end{remark}

\subsection{Moment bounds}
In this subsection we give some upper bounds for the $p$-th moment and the lower bound of the second moment of the solution  for some specific choice of the covariance kernel.
\begin{theorem}\label{T:Rietz}
Assume the following conditions:
\begin{enumerate}[(1)]
 \item  The initial conditions   {satisfy condition (5) of Theorem \ref{T:ExUni};}
 \item $(\alpha,\beta,d)$ satisfies one of the three conditions in \eqref{E:Cases};
 \item $\Lambda(x)=|x|^{-\kappa}$, $x\in\RR^d$ with
\[
  {0<\kappa} < \min (2\alpha-\alpha/\beta, d).
\]
\end{enumerate}
Then the solution $u(t,x)$ to \eqref{E:SPDE} satisfies that for all $p \geq 1$,
\begin{equation}
\EE\left[ \left|u(t,x)\right|^p\right] \leq C^p \widehat{C}_t^p
\exp \left(t (C_{\kappa} C_t \widetilde{C} C_* ({2}/{\nu})^{\kappa/\alpha} (2\pi)^{-d})^{\frac{\alpha}{2\alpha \beta- \alpha -\beta \kappa}} p ^{\frac{2\alpha\beta-\beta \kappa}{2\alpha \beta -\alpha -\beta \kappa}}  \right)\,,
\label{e.p-moment}
\end{equation}
where $C_t$  and $\widehat{C}_t$ are defined in \eqref{E:Ct}   {and \eqref{E:HatCt}, respectively}, $C=C(\alpha,\beta,\kappa)>0$, and
\[
 C_* = \Gamma(2\beta -1-\beta \kappa/\alpha)\quad\text{and}\quad
 \widetilde{C} = \int_{\RR^d} E_{\beta,\beta}^2(-|\xi|^{\alpha}) |\xi|^{\kappa-d} d\xi,
\]
and $C_{\kappa}$ appears in the Fourier transform of $|x|^{-\kappa}$, i.e., $\mu(d\xi) = C_{\kappa}|\xi|^{\kappa-d}$.

In particular, if $\gamma$ is the Dirac delta function
and if the initial data $u_0(x)\equiv u_0>0$ is a constant and $u_1\equiv 0$ when $\beta>1$, then
for some constant $c=c(\alpha,\beta,\kappa)>0$,
\begin{equation}
\EE\left[|u(t,x)|^2\right] \geq c \:   {u_0^2}
\exp \left (t \: (C_{\kappa} \widetilde{C} (4\pi)^{-d} C_*\: (2/\nu)^{\kappa/\alpha})^{\frac{1}{2\beta -1 -\beta \kappa/\alpha} } \right)\,.
\end{equation}
\end{theorem}

The proof of this theorem is given in Section  \ref{Sec:Rietz}.
The same method can be used to obtain  the moment bound for the solution to
the smoothed equation \eref{E:SPDE-I}.

\begin{remark} When  $\beta=1$ and $\alpha=2$, equation \eqref{E:SPDE} is reduced to the
multiplicative stochastic heat equation   (1.1)  considered in \cite{HHNT}.  In this case 
the exponent of $p$ in \eqref{e.p-moment} becomes 
\[
\frac{2\alpha\beta-\beta \kappa}{2\alpha \beta -\alpha -\beta \kappa}=\frac{4-\kappa}{2-\kappa}\:,
\]
which is the same as in \cite[ Theorem 6.1, inequality (6.1)]{HHNT} (with $\kappa=a$). If we assume $\ga(t)=t^{-\tilde \beta}$,
then $C_t=Ct^{-\tilde \beta+1}$.  The exponent  of $t$ in \eqref{e.p-moment} is
\[
1+(-\tilde \beta+1)\left(\frac{\alpha}{2\alpha \beta- \alpha -\beta \kappa}\right)=\frac{4-2\tilde \beta-\kappa}{2-\kappa},
\]
which is the same exponent of $t$ as in \cite{HHNT}, inequality (6.1).  Hence, we conjecture that
the bound \eqref{e.p-moment} is sharp.
\end{remark}

\begin{theorem}
Assume the conditions (1) and (2) of Theorem \ref{T:Rietz}, and assume
\begin{enumerate}
 \item[(3')] $\Lambda(x)=|x|^{-\kappa}$, $x\in\RR^d$ with
\[
  {0< \kappa }< \min (\alpha/\beta, d).
\]
\end{enumerate}
Then the solution $u(t,x)$ to the smoothed equation \eqref{E:SPDE-I} satisfies that for all $p \geq 1$,
\begin{equation}
\EE\left[ u(t,x)^p\right] \leq C^p \widehat{C}_t^p \exp \left(t \: \left[C_{\kappa} C_t \bar{C} C_{\#} (2/\nu)^{\kappa/\alpha} \right]^{\frac{\alpha}{2 \alpha \Ceil{\beta} -\alpha -\beta \kappa }} p^{\frac{2\alpha \Ceil{\beta} -\beta \kappa}{2 \alpha \Ceil{\beta}- \alpha -\beta \kappa}} \right)\,,
\end{equation}
where $C=C(\alpha,\beta,\kappa)>0$,   {$\widehat{C}_t$ is defined in \eqref{E:HatCt},}
\[
C_{\#} = \Gamma(2 \Ceil{\beta}-1-\beta \kappa/\alpha)
\quad\text{and}\quad
\bar{C} = \int_{\RR^d} E_{\beta,\Ceil{\beta}} (- |\eta|^{\alpha}) |\eta|^{\kappa-d} d\eta,
\]
and $C_{\kappa}$ is as defined in Theorem \ref{T:Rietz}.
In particular, if $\gamma$ is the Dirac delta function
and if the initial data $u_0(x)\equiv u_0>0$ is a constant and $u_1\equiv 0$ when $\beta>1$, then for
some constant $c=c(\alpha,\beta,\kappa)>0$,
\begin{equation}
\EE \left[|u(t,x)|^2\right] \geq c \:   {u_0^2} \exp \left(  t\: \left[C_{\kappa} \bar{C} (4\pi)^{-d} C_{\#} (2/\nu)^{\kappa/\alpha} \right]^{\frac{1}{2\beta -1 -\beta \kappa/\alpha} } \right)\,.
\end{equation}
\end{theorem}

The proof of this theorem is a line-by-line change of the proof of Theorem \ref{T:Rietz},
and we leave it to the interested reader.

\subsection{Case $1<\beta<2$ and $d\ge 2$}\label{Subsec:1beta2}

When $1<\beta<2$ and $\al\not=2$,  we could not   show the nonnegativity of $Y(t, x)$ for high dimension ($d\ge 2$)
(see Theorem \ref{T:PDE} (b)). However,  with a slightly different approach,
it is possible to obtain  similar results 
to Theorem \ref{T:ExUni} for the Riesz kernel case. 

\begin{theorem}\label{beta.2}
Assume the conditions (2), (3) and (5) of Theorem \ref{T:ExUni}, and assume 
\begin{enumerate}
\item[(4')] $\Lambda(x)=|x|^{-\kappa}$, $x\in\RR^d$ with
\[
0<\kappa< \min (2\alpha-\alpha/\beta, d).
\]
\end{enumerate}
Then relation \eref{eq: L2 chaos} holds for each $(t,x)$.
Consequently, equation \eref{E:SPDE} admits a unique mild solution in the sense of Definition \ref{D:Sol}.
\end{theorem}

  {This theorem is proved in Section \ref{Sec:1beta2}.}

\begin{remark} It is easy to see that  the condition
$\Lambda(x)=|x|^{-\kappa}$  with
$ 0<\kappa< 2\alpha-\alpha/\beta$ implies Dalang's condition  \eqref{E:DLcond}. Condition $\kappa<d$
is to guarantee that $\La$ is a positive  function.
\end{remark}

\section{Fox H-functions: Some proofs}
\subsection{Proof of Theorem \ref{T:PDE}}\label{Sec:PDE}
The proof of Theorem \ref{T:PDE} will be based on following  lemmas.

\begin{lemma}\label{L:FZ}
The function $Z_{\alpha,\beta,d}(t,x)$ has the Fourier transform given by \eqref{E:FZ}.
\end{lemma}
\begin{proof}
The proof needs relation \eqref{E:ML-H} between the Mittag-Leffler function and the Fox H-function.
We first note some special cases. The case where $\beta\in(0,1]$, $\alpha=2$ and $d\in\NN$ can be found in \cite[Section 4]{Koc} or \cite{SchWyss89}.
For $\beta\in(0,1]$ and for general $\al$, one can simply replace $|\xi|^2$ by $|\xi|^\alpha$ in the argument  of
\cite[Section 4]{Koc} and then use \eqref{E:ML-H} to obtain \eqref{E:FZ}.
The case where $d=1$, $\beta\in (0,2)$, and $\alpha\in (0,2)$ is proved  by \cite{MLP}.
For the general  case, $Z_{\alpha,\beta,d}$ solves
\[
\begin{cases}
\displaystyle
\left(\partial^\beta  + \frac{\nu}{2} (-\Delta)^{\alpha/2} \right) u(t,x)=0, & t>0,\: x\in\R,\\
u(0,x)=\delta_0(x),&\text{if $m\in (0,1)$,}\\
u(0,x) = 0,\quad \left.\frac{\partial}{\partial t} u(t,x)\right|_{t=0}= \delta_0(x),&\text{if $m\in [1,2)$.}
\end{cases}
\]
Hence, the Fourier transform of $Z_{\alpha,\beta,d}$ satisfies
\[
\partial^\beta  \cF Z(t,\cdot)(\xi)= -\frac{\nu}{2} |\xi|^\alpha \cF Z(t,\cdot)(\xi),\quad
\left.\frac{\partial^m}{\partial t^m} \cF Z(t,\cdot)(\xi) \right|_{t=0}= 1.
\]
This equation can be solved explicitly (see, e.g., \cite[Theorem 7.2, on p. 135]{Die04}) as
\[
\cF Z(t,\cdot)(\xi) = I_t^{m} E_{\beta}(-\nu |\xi|^\alpha t^\beta/2 ),
\]
which gives immediately \eqref{E:FZ} when $m=0$. When $m=1$,
the integral can be evaluated by \cite[(1.99)]{Podlubny99FDE} to
give
\[
\cF Z(t,\cdot)(\xi) =t E_{\beta,2}(-\nu |\xi|^2 t^\beta/2).
\]
This completes the proof of Lemma \ref{L:FZ}.
\end{proof}

\begin{lemma}\label{L:ZExp}
 The function $Z_{\alpha,\beta,d}(t,x)$ can be expressed in \eqref{E:Zab}.
\end{lemma}
\begin{proof}
Following Lemma \ref{L:FZ}, we need to compute the inverse Fourier transform of
\eqref{E:FZ}.
Instead of finding the inverse Fourier transform, it turns out that it is easier to
verify  that the Fourier transform
of \eqref{E:Zab} is equal to the right hand side of \eqref{E:FZ}.  Let now $Z$ be defined by
\eqref{E:Zab}.

{\bigskip\noindent\bf Case I $d=1$.}\quad
Notice that $x\mapsto Z(t,x)$ is an even function. We have that
\[
\cF Z(t,\cdot)(\xi)=2 \pi^{-1/2}t^{\Ceil{\beta}-1} \int_0^\infty \ud x\:
x^{-1}
\FoxH{2,1}{2,3}{\frac{ x^\alpha}{2^{\alpha-1}\nu t^\beta}}{(1,1),\:(\Ceil{\beta},\beta)}{(1/2,\alpha/2),\:(1,1),\:(1,\alpha/2)}
\cos(x \xi).
\]
Write the $\cos(\cdot)$ function in the Fox H-function form by \eqref{E:cos} and then apply
Property \ref{Prop:Power} to both Fox H-functions and Property \ref{Prop:frac} to the Fox H-function coming from $\cos(\cdot)$:
\begin{align*}
\cF Z_{\alpha,\beta}(t,\cdot)(\xi)= \alpha^{-1} t^{\Ceil{\beta}-1}\int_0^\infty \frac{\ud x}{x}\: &
\FoxH{2,1}{2,3}{\frac{x}{2^{1-1/\alpha}\nu^{1/\alpha} t^{\beta/\alpha}}}
{(1,1/\alpha),\:(\Ceil{\beta},\beta/\alpha)}{(1/2,1/2),\:(1,1/\alpha),\:(1,1/2)}\\
&\times \FoxH{0,1}{2,0}{\frac{2}{x |\xi|}}{(1,1/2),\;(1/2,1/2)}{\midrule }.
\end{align*}
Now we will apply Theorem \ref{T:HConvH}.
Notice that both condition (1):
$a_1^*=(2-\beta)/\alpha>0$, $a_2^*=0$, $\Delta_2=-1\ne 0$,
and condition \eqref{E:HConvH}: 
\[
A_1=0, \quad A_2 =\infty,\quad B_1=\min(1,\alpha),\quad B_2=0,
\]
of Theorem \ref{T:HConvH} hold. 
Hence, Theorem \ref{T:HConvH} implies that
\[
\cF Z(t,\cdot)(\xi)=
\alpha^{-1}t^{\Ceil{\beta}-1}
\FoxH{2,2}{4,3}{\frac{1}{2^{-1/\alpha}\nu^{1/\alpha}t^{\beta/\alpha}}}{(1,1/\alpha),\:(1,1/2),\:(1/2,1/2),\:(\Ceil{\beta},\beta/\alpha)}{(1/2,1/2),\: (1,1/\alpha),\:(1,1/2)}.
\]
Then apply Properties \ref{Prop:Power} and \ref{Prop:Reduction} to simplify the above expression:
\begin{align*}
\cF Z(t,\cdot)(\xi)
&= t^{\Ceil{\beta}-1}
\FoxH{2,2}{4,3}{2\left(\nu t^{\beta} |\xi|^\alpha\right)^{-1}}{
(1,1),\:(1,\alpha/2),\:(1/2,\alpha/2),\:(\Ceil{\beta},\beta)
}{
(1/2,\alpha/2),\:(1,1),\:(1,\alpha/2)}\\
&=t^{\Ceil{\beta}-1}\FoxH{1,1}{1,2}{2^{-1}\nu t^\beta |\xi|^\alpha}{(0,1)}{(0,1),\:(1-\Ceil{\beta},\beta)}\,.
\end{align*}
This proves the lemma when $d=1$.

{\bigskip\noindent\bf Case II $d\ge 2$.}\quad
Because the function $x\mapsto Z_{\alpha,\beta,d}(t,x)$ is a radial function,
it is known that (see, e.g., \cite[Theorem 3.3 on p. 155]{SteinWeiss71}),
\[
\cF Z(t,\cdot)(\xi)=2^{d/2} t^{\Ceil{\beta}-1} |\xi|
\int_0^\infty \ud x \:
\FoxH{2,1}{2,3}{\frac{ x^\alpha}{2^{\alpha-1}\nu t^\beta}}{(1,1),\:(\Ceil{\beta},\beta)}{(1/2,\alpha/2),\:(1,1),\:(1,\alpha/2)}
J_{(d-2)/2}(x|\xi|) (|\xi| x)^{-d/2},
\]
where $J_\nu(x)$ is the Bessel function of the first kind.
Then we can apply Theorem \ref{T:Hankel}.
Similar to the previous case, all conditions are satisfied with
the condition $a^*=2-\beta>0$.
Hence,
\begin{align*}
\cF Z(t,\cdot)(\xi) &=
t^{\Ceil{\beta}-1}\FoxH{2,2}{4,3}{2\left(\nu t^{\beta} |\xi|^\alpha\right)^{-1}}{
(1,1),\:(1,\alpha/2),\:(d/2,\alpha/2),\:(\Ceil{\beta},\beta)
}{
(d/2,\alpha/2),\:(1,1),\:(1,\alpha/2)
}\\
&=t^{\Ceil{\beta}-1} \FoxH{1,1}{1,2}{2^{-1}\nu t^\beta |\xi|^\alpha}{(0,1)}{(0,1),\:(1-\Ceil{\beta},\beta)}\,,
\end{align*}
where the second equality is due to Property \ref{Prop:Reduction}.
This completes the proof of Lemma \ref{L:ZExp}.
\end{proof}

Note that the case $d\ge 2$ can be also proved by writing the Bessel function in the Fox H-function form
by \eqref{E:BesselJ} and then applying Theorem \ref{T:HConvH}. We leave the details for interested readers.

\begin{lemma}\label{L:YExp} The fundamental  solutions
$Y_{\alpha,\beta,d}(t,x)$ and $Z^*_{\alpha,\beta,d}(t,x)$ are given
by  \eqref{E:Yab} and \eqref{E:Z*ab}, respectively.
\end{lemma}
\begin{proof}
We first prove the expression for $Y_{\alpha,\beta,d}$.
By Section 2 of \cite{EK}, we know that $Y_{\alpha,\beta,d}(t,x)$ is
the Riemann-Liouville fractional derivative in $t$ of $Z_{\alpha,\beta,d}(t,x)$ of order $\Ceil{\beta}-\beta$.
Notice that $Z_{\alpha,\beta,d}(0,x)=0$ for $|x|\ne 0$.
Denote the {\it Riemann-Liouville derivative} of order $\beta\in (0,1)$ by $D_{0+}^\beta$, i.e.,
\begin{align}\label{E:RL-D}
 \left(D_{0+}^\beta f\right)(t)=\frac{\ud }{\ud t} \frac{1}{\Gamma(1-\beta)}\int_0^t\frac{f(x)}{(t-x)^\beta}\ud x,\quad \text{for $t>0$.}
\end{align}
By  Property \ref{Prop:frac},
\[
Z_{\alpha,\beta,d}(t,x)=
\pi^{-d/2} t^{\Ceil{\beta}-1}|x|^{-d}
\FoxH{1,2}{3,2}
{\frac{2^{\alpha-1}\nu t^\beta}{|x|^\alpha}}
{(1-d/2,\alpha/2),\:(0,1),\:(0,\alpha/2)}
{(0,1),\:(1-\Ceil{\beta},\beta)}\,.
\]
Because $a^* = (2-\beta) +(2-\alpha)/2>0$, we can apply Theorem \ref{T:RL-D},
\begin{align*}
D_{0+}^{\Ceil{\beta}-\beta} Z_{\alpha,\beta,d}(t,x)&=
\pi^{-d/2} |x|^{-d} t^{\beta-1}
\FoxH{1,3}{4,3}
{\frac{2^{\alpha-1}\nu t^\beta}{|x|^\alpha}}
{(1-\Ceil{\beta},\beta),\:(1-d/2,\alpha/2),\:(0,1),\:(0,\alpha/2)}
{(0,1),\:(1-\Ceil{\beta},\beta),\:(1-\beta,\beta)}.
\end{align*}
Then we  use Properties \ref{Prop:Reduction} and \ref{Prop:Power}  to simplify the above expression
to obtain \eqref{E:Yab}.
The expression for $Z_{\alpha,\beta,d}^*$ can be proved in  a similar  way.
\end{proof}

\begin{lemma}\label{L:FY}
The Fourier transforms of $Y_{\alpha,\beta,d}(t,x)$ and $Z^*_{\alpha,\beta,d}(t,x)$
are given by  \eqref{E:FY} and \eqref{E:FZ*}, respectively.
\end{lemma}
\begin{proof}
We first consider $Y_{\alpha,\beta,d}$.
From  Lemma \ref{L:FZ} and   the proof of Lemma \ref{L:YExp} it follows
\[
\mathcal{F}Y_{\alpha,\beta,d}(t,\cdot)(\xi) = D_{0+}^{\Ceil{\beta}-\beta}
t^{\Ceil{\beta}-1}\FoxH{1,1}{1,2}{2^{-1}\nu t^\beta |\xi|^\alpha}{(0,1)}{(0,1),\:(1-\Ceil{\beta},\beta)}\,.
\]
Because $a^*=2-\beta>0$, we can apply Theorem \ref{T:RL-D} to obtain that
\[
\mathcal{F}Y_{\alpha,\beta,d}(t,\cdot)(\xi) =
t^{\beta-1}\FoxH{1,2}{2,3}{2^{-1}\nu t^\beta |\xi|^\alpha}
{(1-\Ceil{\beta},\beta),\:(0,1)}{(0,1),\:(1-\Ceil{\beta},\beta),\:(1-\beta,\beta)}\,.
\]
This is simplified to \eqref{E:FY}  by Properties \ref{Prop:Reduction} and \ref{Prop:Power}.
The identity \eqref{E:FZ*}  can be obtained  in a similar  way.
\end{proof}

\begin{lemma}\label{L_:NonNeg}
For all $\mu>0$ and $0<\theta<\min(1,\mu)$, the following H-function is nonnegative:
\begin{align}\label{E_:H1}
\FoxH{1,0}{1,1}{|x|}{(\mu,\theta)}{(1,1)}\ge 0\,,  \quad \text{for all $\ x\in \RR$}\,. 
\end{align}
\end{lemma}
\begin{proof}
We only need to prove that the following function is nonnegative
\[
f(x) = |x|^{-1} \FoxH{1,0}{1,1}{|x|}{(\mu,\theta)}{(1,1)}, \quad x\in\RR.
\]
By Theorem \ref{T:Laplace} and equation \eqref{E:ML-H}, the Laplace transform of $f$ is equal to
\[
\int_0^\infty \ud x \: e^{-xz} f(x) = E_{\theta,\mu}(-z).
\]
By \cite{Sch96CM}, we know that the above Mittag-Leffler function $E_{\alpha,\beta}(-z)$ is completely monotonic
if and only if $0<\alpha\le \min(\beta,1)$.
Then the Bernstein  theorem (see, e.g., \cite[Theorem 12a]{Widder41LT}) implies that the function $f(x)$ is nonnegative.
\end{proof}

\begin{lemma}\label{L:NonNeg}
The nonnegative statements in   Theorem \ref{T:PDE} hold  true.
\end{lemma}
\begin{proof}
We first prove the case (a).  In this case, $\beta\in (0,1]$.
When $\beta=1$, $Z$ and $Y$ coincide and they are alpha stable densities.
Hence, we need only consider the case $\beta\in (0,1)$.
Because $\lim_{t\rightarrow 0}Z_{\alpha,\beta,d}(t,x)=0$ for all $|x|\ne 0$,
and noticing that $I_t^{1-\beta}\partial ^{1-\beta}=\mathop{Id}$ (see, e.g., Theorem 3.8 of \cite{Die04}), we see that
\[
Z_{\alpha,\beta,d}(t,x)= I_t^{1-\beta}\partial ^{1-\beta} Z_{\alpha,\beta,d}(t,x)
=I_t^{1-\beta}Y_{\alpha,\beta,d}(t,x).
\]
Hence, it suffices  to show  the nonnegativity of $Y_{\alpha,\beta,d}(t,x)$.
Notice that by Property \ref{Prop:Power},
\begin{align*}
J:=& \int_0^\infty\frac{\ud s}{s}\:
\FoxH{1,1}{1,2}{\frac{|x|^\alpha}{2^{\alpha-1}\nu}s^\beta}{(1,1)}{(d/2,\alpha/2),\:(1,\alpha/2)}
\FoxH{1,0}{1,1}{(ts)^{-\beta}}{(\beta,\beta)}{(1,1)}\\
=& \beta^{-2}
\int_0^\infty\frac{\ud s}{s}\:
\FoxH{1,1}{1,2}{\frac{|x|^{\alpha/\beta}}{2^{(\alpha-1)/\beta}\nu^{1/\beta}}s}{(1,1/\beta)}{(d/2,\alpha/2\beta),\:(1,\alpha/2\beta)}
\FoxH{1,0}{1,1}{\frac{1}{ts}}{(\beta,1)}{(1,1/\beta)}.
\end{align*}
Now we check conditions in Theorem \ref{T:HConvH}.
Condition (1) of Theorem \ref{T:HConvH} is satisfied because $a_1^*=1/\beta>0$ and $a_2^*=(1-\beta)/\beta>0$.
Condition \eqref{E:HConvH} of Theorem \ref{T:HConvH} holds because
\[
A_1=0,\quad
A_2=\beta,\quad
B_1= d\beta/\alpha,\quad
B_2 = \infty.
\]
Hence, Theorem \ref{T:HConvH} implies that
\[
J=\beta^{-2}
\FoxH{2,1}{2,3}{\frac{|x|^{\alpha/\beta}}{t 2^{(\alpha-1)/\beta}\nu^{1/\beta}}}{(1,1/\beta),\:(\beta,1)}{(d/2,\alpha/2\beta),\:(1,1/\beta),\:(1,\alpha/2\beta)}.
\]
Then by Property \ref{Prop:Power}, we see that 
\begin{align*}
Y(t,x)
&=\beta \pi^{-d/2}t^{\beta-1}|x|^{-d}\int_0^\infty\ud s
\: s^{-1}
\FoxH{1,1}{1,2}{\frac{|x|^\alpha}{2^{\alpha-1}\nu}s^\beta}{(1,1)}{(d/2,\alpha/2),\:(1,\alpha/2)}
\FoxH{1,0}{1,1}{(ts)^{-\beta}}{(\beta,\beta)}{(1,1)}.
\end{align*}
By Lemma \ref{L_:NonNeg}, the second H-function in the above equation is nonnegative. On the other hand,
Theorem \ref{T:Stable} tells us that  the first H-function is nonnegative.
Thus, $Y(t,x)$ is nonnegative.

As for the case (b),     it is known from \cite{Pskhu09} that $Y_{2,\beta,d}$ is nonnegative for $d\le 3$.
By the same argument as in the proof of (a), $Z_{2,\beta,d}$ is also nonnegative.

Finally, for the case (c), it is proved in \cite{MLP} that   $Z^*_{\alpha,\beta,1}(t,x)$ is nonnegative.
By the same reason as in the proof of (a), $Y_{\alpha,\beta,1}$ and $Z_{\alpha,\beta,1}$ are
fractional integrals of $Z^*_{\alpha,\beta,1}$ of   orders  $1-\beta$ and $1$, respectively. Therefore,
both $Y_{\alpha,\beta,1}$ and $Z_{\alpha,\beta,1}$ are nonnegative as well.
The proof of Lemma \ref{L:NonNeg} is now complete.
\end{proof}

\begin{proof}[Proof of Theorem \ref{T:PDE}] The  Theorem \ref{T:PDE} follows from the above lemmas.
\end{proof}

\subsection{Proof of Theorem \ref{T:Stable}}\label{Sec:Stable}
\begin{proof}[Proof of Theorem \ref{T:Stable}]
The characteristic function \eqref{E_:Stable} of $X$ is proved in \cite[(7.5.3) on p. 211]{UchaikinZolotarev99}.
For the density $\rho_{\alpha,d}$, we need to compute the inverse Fourier transform.
From   \cite[(7.5.5)]{UchaikinZolotarev99} this inverse transform is
\[
\rho_{\alpha,d}(r) = (2\pi)^{-d/2}r^{1-d/2}\int_0^\infty e^{-t^\alpha} J_{(d-2)/2}(r t) t^{d/2}\ud t\,.
\]
By \eqref{E:BesselJ} and \eqref{E:zExp}, we have that
\[
t^{(d+2)/2}J_{(d-2)/2}(r t) =
(2/r)^{(d+2)/2}
\FoxH{1,0}{0,2}{\frac{r^2 t^2}{4}}{\midrule}{(d/2,1),\:(1,1)}
\]
and
\[
e^{-t^\alpha}= \frac{1}{\alpha}
\FoxH{1,0}{0,1}{t}{\midrule}{(0,1/\alpha)}.
\]
Hence,
\[
\rho_{\alpha,d}(r) = \pi^{-d/2} r^{-d} \int_0^\infty
t^{-1}
\FoxH{1,0}{0,2}{\left(\frac{rt}{2}\right)^{\alpha}}{\midrule}{(d/2,\alpha/2),\:(1,\alpha/2)}
\FoxH{1,0}{0,1}{t}{\midrule}{(0,1/\alpha)} \ud t.
\]
Application of Theorem \ref{T:HConvH} to evaluate the above integral yields the theorem.
\end{proof}

\section{Proof of Theorem \ref{T:ExUni}}\label{Sec:ExUni}

\begin{proof}[Proof of Theorem \ref{T:ExUni}]
Recall that $J_0(t,x)$ defined by  \eqref{E:J0} is the solution to the homogeneous equation.
Using an iteration procedure as in \cite{HHNT},  we have
\begin{align*}
 f_n(s_1, x_1,\cdots, s_n, x_n, t, x)=g_n(s_1, x_1,\cdots, s_n, x_n, t, x)J_0(s_{\sigma(1)}, x_{\sigma(1)})
\end{align*}
where
\[
g_n(s_1, x_1,\cdots, s_n, x_n, t, x)=\frac{1}{n!}Y(t-s_{\sigma(n)}, x-x_{\sigma(n)} )\cdots Y(s_{\sigma(2)}-s_{\sigma(1)}, x_{\sigma(2)}-x_{\sigma(1)}) \,,
\]
and $\sigma$ denotes a permutation of  $ \{1,2,\cdots, n\}$ such that
$0<s_{\sigma(1)}<\cdots<s_{\sigma(n)}<t$.
Fix $t>0$ and $x\in \mathbb{R}^d$, set $ f_n(s, y, t, x)=f_n(s_1, y_1,\cdots, s_n, y_n, t, x).$
Then we have that
\begin{multline}\label{E:fnNorm}
n! \| f_n(\cdot,\cdot,t,x)\|^2_{\mathcal{H}^{\otimes n}}\\
= n!\int_{[0,t]^{2n}} \ud s\ud r\int_{\mathbb{R}^{2nd}}\ud y\ud z\: f_n(s, y, t, x)f_n(r, z, t, x)\prod_{i=1}^n\Lambda(y_i-z_i)\prod_{i=1}^n\gamma(s_i-r_i).
\end{multline}
where $\ud y=\ud y_1 \cdots \ud y_n$, the differentials $\ud z$, $\ud s$ and $\ud r$ are defined similarly.
Set $\mu(\ud\xi) : = \prod_{i=1}^{n} \mu(\ud\xi_i)$.
Using the Fourier transform and Cauchy-Schwartz inequality together with \eref{Davar id}, we obtain that
\begin{align}
\label{E:n!fn}
n!\| f_n(\cdot,\cdot, t, x)\|^2_{\mathcal{H}^{\otimes n}}
\leq& \frac{\widehat{C}_t^2 \: n!}{(2\pi)^{nd}}  \int_{[0,t]^{2n}}\int_{\mathbb{R}^{nd}}\mathcal{F}g_n(s, \cdot , t, x)(\xi) \overline{\mathcal{F}g_n(r, \cdot, t, x)(\xi)}\mu(\ud\xi)\prod_{i=1}^{n}\gamma(s_i -r_i)\ud s\ud r\\
\notag
\leq&  \frac{\widehat{C}_t^2 \: n!}{(2\pi)^{nd}}  \int_{[0,t]^{2n}}\left(\int_{\mathbb{R}^{nd}} \big(\mathcal{F}g_n(s, \cdot, t, x)(\xi)\big)^2\mu(\ud\xi)\right)^{1/2} \\
\notag
&\times \left(\int_{\RR^{nd}}\big(\mathcal{F}g_n(r, \cdot, t, x)(\xi)\big)^2 \mu(\ud\xi)\right)^{1/2}
\prod_{i=1}^{n}\gamma(s_i -r_i)\ud s\ud r \,,
\end{align}
where the constant $\widehat{C}_t$ is defined in \eqref{E:HatCt}.
Thus, thanks to the basic inequality $ab \leq 2^{-1}(a^2+ b^2)$ and the fact that $\gamma$ is locally integrable, we obtain
\begin{align*}
n!\| f_n(\cdot,\cdot,t, x)\|^2_{\mathcal{H}^{\otimes n}}
&\leq  \frac{\widehat{C}_t^2\:  n!}{(2\pi)^{nd}} \int_{[0,t]^{2n}} \int_{\mathbb{R}^{nd}}|\mathcal{F}g_n(s, \cdot, t, x)(\xi)|^2\mu(\ud\xi)\prod_{i=1}^{n}\gamma(s_i -r_i) \ud s \ud r\\
&\leq  \frac{\widehat{C}_t^2\:  C_t^n n!}{(2\pi)^{nd}} \int_{[0,t]^n}\ud s\int_{\mathbb{R}^{nd}}|\mathcal{F}g_n(s, \cdot, t, x)(\xi)|^2\mu(\ud\xi)\,,
\end{align*}
where the constant $C_t$ is defined in \eqref{E:Ct}. Furthermore, from the Fourier transform of $Y(t,\cdot)$ we can check that
\begin{multline*}
|\mathcal{F}g_n(r, \cdot, t, x)(\xi)|^2\\
=\frac{1}{(n!)^2}\prod_{i=1}^{n}
\bigg[ (s_{\sigma(i+1)}-s_{\sigma(i)})^{\beta-1}E_{\beta, \beta}\big(-2^{-1}\nu(s_{\sigma(i+1)}-s_{\sigma(i)})^{\beta}|\xi_{\sigma(i)}+\cdots +\xi_{\sigma(1)}|^\alpha\big)
\bigg]^2\,,
\end{multline*}
where we have set $ s_{\sigma(n+1)}=t $. As a consequence,
\begin{align}\notag
&\int_{\mathbb{R}^{nd}}|\mathcal{F}g_n(s, \cdot, t, x)(\xi)|^2\mu(\ud\xi)\\ \notag
&\leq \frac{1}{(n!)^2}\prod_{i=1}^{n} \sup_\eta\left|\int_{\mathbb{R}^{d}}(Y(s_{\sigma(i+1)}-s_{\sigma(i)},\cdot)*Y(s_{\sigma(i+1)}-s_{\sigma(i)},\cdot))
(x_{\sigma(i)})e^{i\eta\cdot x_{\sigma(i)}}\Lambda(x_{\sigma(i)})\ud x_{\sigma(i)}\right|\\
&\leq \frac{1}{(n!)^2}\prod_{i=1}^{n} \left|\int_{\mathbb{R}^{d}}(Y(s_{\sigma(i+1)}-s_{\sigma(i)},\cdot)*Y(s_{\sigma(i+1)}-s_{\sigma(i)},\cdot))
(x_{\sigma(i)}) \Lambda(x_{\sigma(i)})\ud x_{\sigma(i)}\right|\notag \\
&\leq \frac{1}{(n!)^2}\prod_{i=1}^{n} \int_{\mathbb{R}^d}\big[(s_{\sigma(i+1)}-s_{\sigma(i)})^{\beta-1}
E_{\beta,\beta}\big(-2^{-1}\nu(s_{\sigma(i+1)}-s_{\sigma(i)})^{\beta}|\xi_{\sigma(i)}|^\alpha\big)\big]^2 \mu(\ud\xi_{\sigma(i)}),
\label{E_:YNonNeg}
\end{align}
where we have used the fact that $ |e^{ix_{\sigma(i)}\cdot{\eta}}|=1$
and  that $Y$  and $\La$ are nonnegative to get rid of the supremum  in $\eta$.
Therefore, using Fourier transform again we have
\begin{equation}
\label{E:Fgn}
\begin{aligned}
n! \| f_n(\cdot,\cdot,t,x)\|^2_{\mathcal{H}^{\otimes n}}
\leq \frac{\widehat{C}_t^2 C_t^n}{(2\pi)^{nd}} & \int_{\mathbb{R}^{nd}}\mu(\ud\xi) \int_{T_n(t)}\ud s  \\
&\times \prod_{i=1}^{n}
\big(s_{i+1}-s_{i}\big)^{2\beta-2}
E_{\beta,\beta}^2\big(-2^{-1}\nu(s_{i+1}-s_{ i})^{\beta}|\xi_{i}|^\alpha\big)\,,
\end{aligned}
\end{equation}
where $T_n(t)$ denotes the simplex
\begin{align}\label{E:Tn}
  {
T_n(t):=\{s=(s_1,\cdots,s_n):\:0<s_1<\cdots<s_n<t\}.
}
\end{align}
By the change of variables $s_{i+1}-s_i=w_i$ for $1 \leq i \leq n-1$ and $t-s_n=w_n$,
we see that
\begin{align*}
n!\| f_n(\cdot,\cdot,t,x)\|^2_{\mathcal{H}^{\otimes n}}
\leq \frac{\widehat{C}_t^2 C_t^n}{(2\pi)^{nd}}\int_{\mathbb{R}^{nd}}\int_{S_{t,n}}\prod_{i=1}^{n} w_i^{2\beta-2}
E_{\beta,\beta}^2\big(-2^{-1}\nu w_i^{\beta}|\xi_{i}|^\alpha\big)\ud w_i \mu(\ud\xi_i),
\end{align*}
where
\[
S_{t,n}=\{(w_1,\cdots,w_n)\in[0, \infty)^n: \: w_1+\cdots+w_n\leq t \}.
\]
We take $N\geq 1 $ which will be chosen later, and let
\begin{align}\label{E:CnDn}
C_N=\int_{|\xi|\geq N}\frac{\mu(d\xi)}{|\xi|^{2\alpha-\alpha/\beta}} \quad \text {and} \quad  D_N=\mu \{\xi\in {\mathbb{R}}^d: |\xi|\leq N \}.
\end{align}
Let $I$ be a subset of $ \{1, 2, \cdots, n\}$ and $I^c=\{1, 2, \cdots, n\} \backslash I$. Then we have
\begin{align*}
\int_{\mathbb{R}^{nd}}\int_{S_{t,n}}&\prod_{i=1}^{n} w_i^{2\beta-2}
E_{\beta,\beta}^2\big(-2^{-1}\nu w_i^{\beta}|\xi_{i}|^\alpha\big)\ud w_i\mu(\ud\xi_i)\\
=&\int_{\mathbb{R}^{nd}}\int_{S_{t,n}}\prod_{i=1}^{n} w_i^{2\beta-2}
E_{\beta,\beta}^2\big(-2^{-1}\nu w_i^{\beta}|\xi_{i}|^\alpha\big)\big(\Indt{|\xi_i|\leq N} +{\bf 1}_{\{|\xi_i|> N\}}\big) \ud w_i\mu(\ud\xi_i)\\
=&\sum_{I \subset \{1, 2, \cdots, n \} }\int_{\mathbb{R}^{nd}}\ud w\int_{S_{t,n}}\mu(\ud\xi)\prod_{i\in I}E_{\beta,\beta}^2\big(-2^{-1}\nu w_i^{\beta}|\xi_{i}|^\alpha\big) w_i^{2\alpha-2}\Indt{|\xi_i|\leq N} \\
& \times \prod_{j\in I^c}E_{\beta,\beta}^2\big(-2^{-1}\nu w_j^{\beta}|\xi_{j}|^\alpha\big)  w_j^{2\beta-2} {\bf 1}_{\{|\xi_j|> N\}}\,.
\end{align*}
where  $\ud w=\ud w_1\cdots \ud w_n$.
For the indices $i$ in the set $I$, for some constant $C_\beta\ge 1$ (one may choose $C_\beta=\Gamma(\beta)^{-2}$)
\begin{align}\label{E:Cbeta}
E_{\beta,\beta}^2\big(-2^{-1}\nu w_i^{\beta}|\xi_{i}|^\alpha\big)\le C_\beta.
\end{align}
Now  using the inclusion $S_{t,n} \subset S_{t}^I \times S_{t}^{I^c}$ with
\begin{equation*}
 S_{t}^I=\bigg\{(w_i,i\in I): w_i\geq 0, \: \sum_{i\in I}w_i\leq t\bigg\}\quad \text{and} \quad
 S_{t}^{I^c}=\bigg\{(w_i,i\in I^c): w_i\geq 0,\: \sum_{i\in I^c}w_i\leq t \bigg \}\,,
\end{equation*}
we obtain that
\begin{align*}
\int_{\mathbb{R}^{nd}}\int_{S_{t,n}} &\prod_{i=1}^{n} w_i^{2\beta-2}
E_{\beta,\beta}^2\big(-2^{-1}\nu w_i^{\beta}|\xi_{i}|^\alpha\big)\ud w_i \mu(\ud\xi_i)\\
\le C_\beta^{|I|} &\sum_{I \subset \{1, 2, \cdots, n \} }\int_{\mathbb{R}^{nd}}\mu(\ud\xi)\int_{S^I_{t} \times S^{I^c}_{t}} \ud w \: \\
&\times \prod_{i\in I}w_i^{2\beta-2} \Indt{|\xi_i|\leq N}
\prod_{j\in I^c}w_j^{2\beta-2 } \Indt{|\xi_j|> N} E_{\beta,\beta}^2\big(-2^{-1}\nu w_j^{\beta}|\xi_{j}|^\alpha\big)\,.
\end{align*}
Furthermore, one can bound the integral over $S^{I^c}_{t}$ in the following way
\begin{eqnarray*}
\int_{  S^{I^c}_{t}}\prod_{j\in I^c}w_j^{2\beta-2 }
E_{\beta,\beta}^2\big(-2^{-1}\nu w_j^{\beta}|\xi_{j}|^\alpha\big)\ud w_j
\leq \int_{ \R_+^{|I^c|}}\prod_{j\in I^c}w_j^{2\beta-2 } E_{\beta,\beta}^2\big(-2^{-1}\nu w_j^{\beta}|\xi_{j}|^\alpha\big)\ud w_j\,.
\end{eqnarray*}
Then make the change of variables  $w_j^{\beta}|\xi_{j}|^\alpha \rightarrow v_j$ to obtain
\begin{align*}
\int_{[0, \infty)^{|I^c|}}\prod_{j\in I^c}w_j^{2\beta-2 } E_{\beta,\beta}^2\big(-2^{-1}\nu w_j^{\beta}|\xi_{j}|^\alpha\big)\ud w_j
&\leq\prod_{j\in I^c} \frac{1}{|\xi_i|^{2\alpha-\alpha/\beta}}\int_0^\infty
\frac{1}{\beta} v_j^{1-1/\beta}E_{\beta,\beta}^2(- 2^{-1}\nu v_j)\ud v_j \\
&\le C_{\nu,\beta}^{|I^c|}\prod_{j\in I^c} \frac{1}{|\xi_i|^{2\alpha-\alpha/\beta}}\,,
\end{align*}
where
\[
C_{\nu,\beta}=\int_0^\infty \frac{1}{\beta}v^{1-1/\beta} E_{\beta,\beta}^2(-2^{-1}\nu v)\ud v.
\]
Note that the integrability of the above quantity at zero and at infinity
    implies that $\beta>1/2$ and $\beta>0$, respectively.
Thus we have the following bound.
\begin{align*}
\int_{\mathbb{R}^{nd}}&\int_{S_{t,n}} \prod_{i=1}^{n} w_i^{2\beta-2}
E_{\beta,\beta}^2\big(-2^{-1}\nu w_i^{\beta}|\xi_{i}|^\alpha\big)\ud w_i\mu(\ud\xi_i)
\\
&\leq\sum_{I \subset \{1, 2, \cdots, n \} }
C_\beta^{|I|}\int_{S_{t}^I}\prod_{i\in I} w_i^{2\beta-2} \ud w_i \cdot
\left(\mu \{\xi\in {\mathbb{R}}^d: |\xi|\leq N \}\right)^{|I|}
C^{|I^c|}_{\nu,\beta}\int_{|\xi_j|>N, \forall j\in I^c}\prod_{j\in I^c}\frac{\mu(\ud\xi_j)}{|\xi_j|^{2\alpha-\alpha/\beta}}\\
&\leq \sum_{I \subset \{1, 2, \cdots, n \} }\frac{C_\beta^{|I|}t^{(2\beta -1)|I|}C_{\nu,\beta}^{|I^c|}}{\Gamma((2\beta -1)|I|+1)} D_N^{|I|} C_N^{n-|I|} \\
&\leq C_*^n \sum_{k=0}^{n}  {n \choose k}  \frac{t^{(2\beta -1)k} }{\Gamma((2\beta -1)k+1)}D_N^k C_N^{n-k}\,.
\end{align*}
where $C_*=\max(C_\beta,C_{\nu,\beta})$, and $C_N$ and $D_N$ are defined in \eqref{E:CnDn}.
Observing the trivial inequality  $ {n \choose k}\le 2^n$, we have
\begin{align*}
\sum_{n=0}^{\infty}n!\| f_n(\cdot,\cdot,t,x)\|^2_{\mathcal{H}^{\otimes n}}
&\leq  \frac{\widehat{C}_t^2}{(2\pi)^{nd}} \sum_{k=0}^{\infty}\sum_{n=k}^{\infty}
{n \choose k} (C_*C_t)^n \frac{t^{(2\beta -1)k }} {\Gamma((2\beta -1)k+1)}D_N^k C_N^{n-k}\\
&\le \frac{\widehat{C}_t^2}{(2\pi)^{nd}} \sum_{k=0}^{\infty}\sum_{n=k}^{\infty}   \frac{t^{(2\beta -1)k} }{\Gamma((2\beta -1)k+1)}D_N^k C_N^{-k}(2C_*C_tC_N)^n\,.
\end{align*}
Choosing  $N$  sufficiently large  so  that $2C_*C_tC_N<1$ yields
\begin{eqnarray*}
\sum_{n=0}^{\infty}n!\| f_n(\cdot,\cdot,t,x)\|^2_{\mathcal{H}^{\otimes n}}
\leq \frac{\widehat{C}_t^2}{(2\pi)^{nd}} \sum_{k=0}^{\infty}  \frac{t^{(2\beta -1)k} }{\Gamma((2\beta -1)k+1)}D_N^k C_N^{-k}\frac{(2C_*C_tC_N)^k}{1-2C_*C_tC_N}<\infty\,.
\end{eqnarray*}
This proves \eref{eq: L2 chaos}, and thus the existence and uniqueness of the solution.
\end{proof}

\section{Proof of Theorem \ref{T:Rietz}} \label{Sec:Rietz}

\begin{lemma}\label{L:J0}
Suppose that the initial conditions $u_k(x) \equiv u_k$ are constant. Then
under the three cases of \eqref{E:Cases}, we have that
 \begin{align}\label{E:J0Const}
 J_0(t,x)=
 \begin{cases}
  u_0 & \text{if $\beta\in (0,1]$},\\
  u_0 +  t^{\beta-1} u_1  & \text{if $\beta\in (1,2)$}.
 \end{cases}
 \end{align}
\end{lemma}
\begin{proof}
By Theorem \ref{T:PDE}, we know that under the   {first two cases} of \eqref{E:Cases},
the fundamental solutions are nonnegative and hence,
\begin{align*}
J_0(t,x)= &\sum_{k=0}^{\Ceil{\beta}-1} u_{  {\Ceil{\beta}-1-k}} \int_{\RR^d} \partial ^k Z(t,x-y) \ud y=
\sum_{k=0}^{\Ceil{\beta}-1}u_{  {\Ceil{\beta}-1-k}} \: \cF \left[\partial ^k Z(t,\cdot)\right](0),
\end{align*}
which is equal to the right hand side of \eqref{E:J0Const}.
As for   {the last case} in \eqref{E:Cases},
because $Z$ is still nonnegative, the contribution by $u_0$ can be computed  in the same way.
However, we do not know whether $Z^*$ is nonnegative,
and thus we cannot use the Fourier transform arguments to compute  the contribution by $u_1$.
Instead, we compute  it directly:
\[
\int_{\RR^{  {d}}} Z_{2,\beta,  {d}}^*(t,x)\ud x = S_{  {d-1}} \pi^{-  {d}/2} \int_0^\infty x^{-1}
\FoxH{2,0}{1,2}{\frac{x^\alpha}{2\nu t^\beta}}{(1,\beta)}{(d/2,1),\: (1,1)}\ud x,
\]
where
\begin{align}\label{E:Sn}
S_{  {d-1}} =\frac{2 \pi^{  {d/2}}}{\Gamma(  {d/2})}.
\end{align}
Then by Theorem \ref{T:Laplace}, we have the following Laplace transform:
\[
g(z):=\int_0^\infty e^{-zx}x^{-1}
\FoxH{2,0}{1,2}{\frac{x^\alpha}{2\nu t^\beta}}{(1,\beta)}{(d/2,1),\: (1,1)}\ud x
=
\FoxH{1,2}{2,2}{2\nu t^\beta z^2}{(1-d/2,1),\:(0,1)}{(0,\alpha),\:(0,\beta)}.
\]
Then by Theorem \ref{T:AsypZero},
\[
g(0) = h_{10}^* = \frac{\Gamma(d/2)}{2}.
\]
Putting these identities  together, we have  that
\[
\int_{\RR^{  {d}}} Z_{2,\beta,  {d}}^*(t,x)\ud x=1.
\]
This completes the proof of Lemma \ref{L:J0}.
\end{proof}

\begin{proof}[Proof of Theorem \ref{T:Rietz}]
Since $\Lambda(x)=|x|^{-\kappa}$, we have $\mu(d\xi)= C_{\kappa} |\xi|^{\kappa -d}$, for some coefficient $C_{\kappa}$; see, e.g., \cite{Ste}.
We begin with the upper bound.
By the hypercontractivity property of the $n$-th chaos, i.e.
\begin{equation}
\|I_n(f_n(\cdot,\cdot,t,x))\|_{L^p(\Omega)} \leq (p-1)^{\frac{n}{2}} \|I_n(f_n(\cdot,\cdot,t,x))\|_{L^2(\Omega)}\,.
\end{equation}
On the other hand,  from the proof of Theorem \ref{T:ExUni} (see \eqref{E:Fgn}) it follows
\begin{align*}
&\hspace{-2em}\|I_n(f_n(\cdot,\cdot,t,x))\|^2_{L^2(\Omega)} =n!\Norm{f_n(\cdot,\cdot,t,x)}_{\cH^{\otimes n}}^2
\\
&\le \frac{\widehat{C}_t^2 C_\kappa^n  C_t^n}{(2\pi)^{nd}} \int_{T_n(t)}  \int_{\RR^{nd}}\prod_{i=1}^n (s_{i+1}-s_i)^{2\beta-2} E_{\beta,\beta}^2 (-2^{-1}\nu (s_{i+1}-s_i)^{\beta}|\xi_i|^{\alpha}) |\xi_i|^{\kappa-d}\ud \xi_i \ud s_i\\
&=  \frac{\widehat{C}_t^2 C_\kappa^n C_t^n (2/\nu)^{\kappa n/\alpha}}{(2\pi)^{nd}}  \int_{T_n(t)}  \int_{\RR^{nd}}\prod_{i=1}^n (s_{i+1}-s_i)^{2\beta-2-\frac{\beta \kappa}{\alpha}} E_{\beta,\beta}^2 (-|\eta_i|^{\alpha}) |\eta_i|^{\kappa-d} \ud\eta_i \ud s_i\\
&= \frac{\widehat{C}_t^2 C_\kappa^n C_t^n (2/\nu)^{\kappa n/\alpha} \widetilde{C}^n}{(2\pi)^{nd}}  \int_{T_n(t)}  \prod_{i=1}^n (s_{i+1}-s_i)^{2\beta-2-\frac{\beta \kappa}{\alpha}}  \ud s_i,
\end{align*}
where   {$\widehat{C}_t$ is defined in \eqref{E:HatCt}},
\[
\widetilde{C} := \int_{\R^d} E_{\beta,\beta}^2(-|\eta|^\alpha) |\eta|^{\kappa-d}\ud \eta
=S_{d-1}\int_0^\infty
E_{\beta,\beta}^2(-t^\alpha) t^{\kappa-1} \ud t,
\]
and $S_{d-1}$ is defined in \eqref{E:Sn}.
According to the property of the Mittag-Leffler function at zero and infinity,
if $0<\kappa< 2\alpha$, then the above constant $\widetilde{C}$ is finite.
Then, under the condition that $\kappa<\alpha(2-1/\beta)$ (this condition implies
$0<\kappa< 2\alpha$), the integration over $\ud s$   can be evaluated
explicitly; see \cite[Lemma 4.5]{HHNT}. Hence,
\[
 \|I_n(f_n(\cdot,\cdot,t,x))\|^2_{L^2(\Omega)} \le \frac{1}{(2\pi)^{nd}}
(C_\kappa C_* C_t\widetilde{C})^n \frac{t^{(2\beta -1 -\frac{\beta \kappa }{\alpha})n} \:(2/\nu)^{\kappa n/\alpha} } { \Gamma((2\beta -1 -\frac{\beta \kappa}{\alpha})n+1)}\,,
\]
where  $
C_*:=\Gamma(2\beta-1-\beta\kappa/\alpha)$.\
Denote
\[
\Theta_t := \frac{1}{(2\pi)^d} C_\kappa C_* C_t \widetilde{C}\: (2/\nu)^{\kappa/\alpha}.
\]
Thus we obtain
\begin{equation}
\|I_n(f_n(\cdot,\cdot,t,x))\|_{L^2(\Omega)} \leq  {\widehat{C}_t} \frac{\Theta_t^{n/2} t^{(\beta -\frac{1}{2} -\frac{\beta \kappa }{2\alpha})n} }{ \Gamma((2\beta -1 -\frac{\beta \kappa}{\alpha})n+1)^{\frac{1}{2}}}\,.
\end{equation}
This bound together with the hypercontractivity implies that
\begin{equation}
\|I_n(f_n(\cdot,\cdot,t,x))\|_{L^p(\Omega)} \leq { \widehat{C}_t }\frac{\Theta_t^{n/2} t^{(\beta -\frac{1}{2} -\frac{\beta \kappa }{2\alpha})n}{(p-1)^{\frac{n}{2}} }}{ \Gamma((2\beta -1 -\frac{\beta \kappa}{\alpha})n+1)^{\frac{1}{2}}}\,.
\end{equation}
Therefore,
\begin{align*}
\|u(t,x)\|_{L^p(\Omega)} &
\leq \sum_{n=0}^{\infty} \|I_n(f_n(\cdot,\cdot,t,x))\|_{L^p(\Omega)}
\leq { \widehat{C}_t }\sum_{n=0}^{\infty}   \frac{\Theta_t^{n/2} t^{\theta n} p^{\frac{n}{2}}}{ \Gamma(2\theta n+1)^{\frac{1}{2}}}.
\end{align*}
where
\begin{align}\label{E:theta}
\theta:=\beta-1/2-\beta\kappa/(2\alpha).
\end{align}
Then by the fact that $\Gamma(1+2x)\ge \Gamma(1+x)^2$ for $x>-1$,
\begin{align*}
\|u(t,x)\|_{L^p(\Omega)}
&\leq  \widehat{C}_t \sum_{n=0}^{\infty}  \frac{\Theta_t^{n/2}  t^{\theta n} p^{\frac{n}{2}}}{ \Gamma(\theta n+1)}
= \widehat{C}_t E_{\theta}\left(\Theta_t^{1/2} t^\theta p^{1/2}\right)\\
&\leq C \widehat{C}_t \exp \left(t (C_{\kappa} C_t \widetilde{C} C_* ({2}/{\nu})^{\kappa/\alpha} (2\pi)^{-d})^{\frac{\alpha}{2\alpha \beta- \alpha -\beta \kappa}} p ^{\frac{\alpha}{2\alpha \beta -\alpha -\beta \kappa}}  \right)\,,
\end{align*}
for some positive constant $C=C(\alpha,\beta,\kappa)$,
where in the last step, we have used the asymptotic property of the Mittag-Leffler function (see, e.g., \cite[Theorem 1.3]{Podlubny99FDE}).

\bigskip
Now we consider the special case when $\gamma$ is the Dirac delta function.
  {
By   Lemma \ref{L:J0} and the assumptions on the initial conditions we have
\[
J_0(t,x)= u_0 + t^{\beta-1} u_1 \Indt{\beta >1} = u_0.
\]
}
From the proof of Theorem \ref{T:ExUni} (see \eqref{E:n!fn} and \eqref{E:Fgn}), we see that
\begin{align*}
\|I_n(f_n(\cdot,\cdot,t,x))\|^2_{L^2(\Omega)} = &\frac{1}{n!} \frac{  {u_0^2} C_{\kappa}^n}{(2\pi)^{nd}} \int_{[0,t]^n}\ud s \int_{\RR^{nd}} \ud \xi \: \prod_{i=1}^n (s_{\sigma(i+1)}-s_{\sigma(i)})^{2\beta-2}\\
  & \times  E_{\beta,\beta}^2 \left(-2^{-1}\nu (s_{\sigma(i+1)}-s_{\sigma(i)})^{\beta} |\xi_{\sigma(i)} + \cdots + \xi_{\sigma(1)} |^{\alpha} \right)  |\xi_i|^{\kappa - d}  \\
 =&  \frac{  {u_0^2} C_{\kappa}^n}{(2\pi)^{nd}} \int_{T_n(t)} \ud s\int_{\RR^{nd}}\ud \xi \: \prod_{i=1}^n (s_{i+1}-s_{i})^{2\beta-2}\\
  & \times  E_{\beta,\beta}^2 \left(-2^{-1}\nu (s_{i+1}-s_{i})^{\beta} |\xi_{i} + \cdots + \xi_{1} |^{\alpha} \right) |\xi_i|^{\kappa - d} \,.
\end{align*}
Then by the change of variable $\xi_i + \cdots \xi_1 = \eta_i$ and replacing $\RR^{nd}$ by $\RR_+^{nd}$, we obtain that
\begin{align*}
\|I_n(f_n(\cdot,\cdot,t,x))\|^2_{L^2(\Omega)} = \frac{  {u_0^2} C_{\kappa}^n}{(2\pi)^{nd}} &\int_{T_n(t)} \int_{\RR^{nd}} \prod_{i=1}^n (s_{i+1}-s_{i})^{2\beta-2}\\
  & \times  E_{\beta,\beta}^2 \left(- 2^{-1}\nu (s_{i+1}-s_{i})^{\beta} |\eta_{i} |^{\alpha} \right) |\eta_i-\eta_{i-1}|^{\kappa - d} \ud \xi_i \ud s_i\\
  \geq   \frac{  {u_0^2} C_{\kappa}^n}{(2\pi)^{nd}} & \int_{T_n(t)} \int_{\RR_+^{nd}} \prod_{i=1}^n (s_{i+1}-s_{i})^{2\beta-2}\\
  & \times  E_{\beta,\beta}^2 \left(- 2^{-1}\nu (s_{i+1}-s_{i})^{\beta} |\eta_{i} |^{\alpha} \right) |\eta_i-\eta_{i-1}|^{\kappa - d} \ud\xi_i \ud s_i \\
  \geq   \frac{  {u_0^2} C_{\kappa}^n}{(2\pi)^{nd}}& \int_{T_n(t)} \int_{\RR_+^{nd}} \prod_{i=1}^n (s_{i+1}-s_{i})^{2\beta-2}\\
  & \times  E_{\beta,\beta}^2 \left(- 2^{-1}\nu (s_{i+1}-s_{i})^{\beta} |\eta_{i} |^{\alpha} \right)  |\eta_i|^{\kappa - d} \ud\xi_i \ud s_i\,,
\end{align*}
where $\eta_0 = 0$.
Then with another change of variable $(\nu/2)^{1/\alpha}(s_{i+1}-s_i)^{\beta/\alpha} \eta_i \rightarrow \eta_i$,
and by the same reasoning as before, we obtain that
\begin{align*}
\|I_n(f_n(\cdot,\cdot,t,x))\|^2_{L^2(\Omega)}
\geq & \frac{  {u_0^2} C_{\kappa}^n}{(2\pi)^{nd}}  \int_{T_n(t)} \int_{\RR_+^{nd}} \prod_{i=1}^n (s_{i+1}-s_{i})^{2\beta-2 - \frac{\beta \kappa}{\alpha}} E_{\beta,\beta}^2 \left(- |\eta_{i} |^{\alpha} \right) |\eta_i|^{\kappa - d} \ud\xi_i \ud s_i \\
= &  \frac{  {u_0^2} C_{\kappa}^n}{(2\pi)^{nd}} \left( \frac{\widetilde{C}}{2^d} \right)^n (2/\nu)^{\frac{\kappa n }{\alpha}}\int_{T_n(t)}  \prod_{i=1}^n (s_{i+1}-s_{i})^{2\beta-2 - \frac{\beta \kappa}{\alpha}}  \ud s_i \\
=& \frac{ t^{n(2\beta -1 -\frac{\beta \kappa}{\alpha})} (2/\nu)^{\kappa n/\alpha}   {u_0^2}  C_{\kappa}^n \widetilde{C}^n (4\pi)^{-nd} C_*^n}{\Gamma(n(2\beta -1 -\frac{\beta \kappa}{\alpha})+1)}\,.
\end{align*}
Therefore, by the asymptotic property of the Mittag-Leffler function,
\begin{align*}
\EE\left[ u(t,x)^2\right] \geq &\sum_{n=0}^{\infty} \frac{   {u_0^2}  \left(C_{\kappa} \widetilde{C} (4\pi)^{-d} C_* \right)^n t^{n(2\beta -1 -\frac{\beta \kappa}{\alpha})} (2/\nu)^{\frac{\kappa n }{\alpha}}}{\Gamma(n(2\beta -1 -\frac{\beta \kappa}{\alpha})+1)} \\
\geq& c\:   {u_0^2} \exp \left ((C_{\kappa} \widetilde{C} (4\pi)^{-d} C_* (2/\nu)^{\kappa/\alpha} )^{\frac{1}{2\beta -1 -\beta \kappa/\alpha} } t \right)\,,
\end{align*}
for some positive constant $c=c(\alpha,\beta,\kappa)$.
This completes the proof of Theorem \ref{T:Rietz}.
\end{proof}

\section{Proof of Theorem \ref{beta.2}}\label{Sec:1beta2}

In this section, $C=C_{\alpha, \beta, \cdots} $ denotes  a positive constant, possibly dependent on $\alpha, \beta, d, \nu, \cdots$.
\begin{lemma}\label{est.H} Assume that $\beta\in (0,2)$, $\alpha>0$ and $d\in\NN$.
  {Then there is a nonnegative constant $C_{\alpha,\beta,d}$ such that for all $0<\zeta<\min(d/\alpha, 2)$,
\[
\left|
\FoxH{2,1}{2,3}{z}
 {(1,1),\:(\beta,\beta)}{(d/2,\alpha/2),\:(1,1),\:(1,\alpha/2)}
\right|
 \leq C_{\alpha, \beta, d} \: \frac{z^\zeta}{z^{\zeta+1}+1},\qquad\text{for all $z \ge 0$.}
\]
}
\end{lemma}

\begin{proof}
We first note that condition \eqref{pole.con1} is satisfied.
Because $a^*=2-\beta>0$, we can apply Theorem \ref{T:AsyInfty} to obtain that
\[
\FoxH{2,1}{2,3}{z}
 {(1,1),\:(\beta,\beta)}{(d/2,\alpha/2),\:(1,1),\:(1,\alpha/2)} {  =}  O(1/z) ; \qquad z \rightarrow \infty.
 \]
As   for small $z$,  note that the poles of $\Gamma(1+s)$ are
\[
A:=\{-(1+k): k=0,1,2,\cdots\};
\]
and those of $\Gamma(d/2+\alpha s/2)$ are
\[
B:=\left\{-\frac{2l+d}{\alpha}: l=0,1,2,\cdots\right\}.
\]
To find the leading term when  $z\rightarrow 0$,   we need to find the first nonvanishing residue of $\cH^{2,1}_{2,3}(s) z^{-s}$ at poles $A\cup B$, where
\begin{align}\label{E:Hs}
\cH^{2,1}_{2,3}(s) = \frac{\Gamma(d/2+\alpha s/2)\Gamma(1+s)\Gamma(-s)}{\Gamma(\beta+\beta s)\Gamma(-\alpha s/2)}.
\end{align}

{\bigskip\bf\noindent Case I.~~}
When $d\ne \alpha$ and $d\ne 2\alpha$, then the leading pole ($l=0$) in $B$ does not coincide with the first two poles ($k=0,1$) of $A$. Hence,
the asymptotic expansion in Theorem \ref{T:AsypZero} (1) implies that
\[
 \FoxH{2,1}{2,3}{z}
 {(1,1),\:(\beta,\beta)}{(d/2,\alpha/2),\:(1,1),\:(1,\alpha/2)}
 {  =}  h_{10}^* z^{d/\alpha} +  h_{20}^* z + h_{11}^* z^2 + O(z^2) +O(z^{d/\alpha}) ;\qquad z \rightarrow 0.
\]
Here $h_{ij}^*$ are defined in \eqref{E:h*jl}.
Notice that $h_{20}^*=0$ due to the presence of the parameter $(\beta,\beta)$.
Hence,
\[
\FoxH{2,1}{2,3}{z}
 {(1,1),\:(\beta,\beta)}{(d/2,\alpha/2),\:(1,1),\:(1,\alpha/2)}
 {  =}  O(z^2)+O(z^{d/\alpha}) ;\qquad z \rightarrow 0.
\]

{\bigskip\bf\noindent Case II.~~}
Now we consider the case when $d=\alpha$.
The first pole in $A\cup B$ is $s=-1$, which is of order $2$.
%
Now we compute  the mentioned residue:
\begin{align*}
\mathop{\text{Res}}_{s=-1}\left[\cH_{2,3}^{2,1}(s)z^s \right]
&=\lim_{s\rightarrow -1}\left[(s+1)^2 \cH_{2,3}^{2,1}(s) z^{-s}\right]'\\
&=\lim_{s\rightarrow -1}\left[(s+1)^2 \cH_1^*(s) \cH_2^*(s) z^{-s}\right]'\\
&=
\lim_{s\rightarrow -1}
 z^{-s} \left[\cH_1^*(s)'\cH_2^*(s)+\cH_1^*(s)\cH_2^*(s)' - \cH_1^*(s)\cH_2^*(s) \log z \right],
\end{align*}
where $\cH_{2,3}^{2,1}(s)$ is defined in \eqref{E:Hs} and
\[
\cH_1^*(s)=(s+1)^2\Gamma((1+s)d/2)\Gamma(1+s)
\quad
\text{and}
\quad
\cH_2^*(s) = \frac{\Gamma(-s)}{\Gamma(\beta+\beta s)\Gamma(-ds/2)}.
\]
Now simple calculations show that
\begin{align*}
\cH_1^*(-1) &= \lim_{s\rightarrow -1} \cH_1^*(s)=\frac{2}{d}= \lim_{s\rightarrow-1}\frac{(1+s)^2}{((1+s)d/2)(1+s)}=\frac{2}{d},\\
\cH_2^*(-1) &=\lim_{s\rightarrow -1} \cH_2^*(s) = 0,\\
\left.\frac{\ud }{\ud s}\cH_2^*(s)\right|_{s=-1} &= \lim_{s\rightarrow -1} \frac{\Gamma(-s)}{\Gamma(-ds/2)} \left(\frac{1}{\Gamma(\beta(1+ s))}\right)'\\
&=\frac{\Gamma(1)}{\Gamma(d/2)}\lim_{s\rightarrow -1}- \frac{\psi(\beta(1+ s))}{\Gamma(\beta(1+ s))}\\
&=\frac{\beta}{\Gamma(d/2)}\,,
\end{align*}
where $\psi(z)$ is the digamma function and the last limit is due to (5.7.6) and (5.7.1) of \cite{NIST2010}.
Thus,
\begin{align*}
\cH_1^*(-1)\cH_2^*(-1) =
\cH_1^*(-1)' \cH_2^*(-1)=0\quad\text{and}\quad
\cH_1^*(-1) \cH_2^*(-1)' = \frac{2\beta}{d\Gamma(d/2)} =\frac{\beta}{\Gamma(1+d/2)}.
\end{align*}
Hence,
\[
\mathop{\text{Res}}_{s=-1}\left[\cH_{2,3}^{2,1}(s)z^s \right]=\frac{\beta\: z}{\Gamma(1+d/2)}.
\]
Therefore, by the definition of the Fox H-function,
\[
\FoxH{2,1}{2,3}{z}
 {(1,1),\:(\beta,\beta)}{(d/2,d/2),\:(1,1),\:(1,d/2)}
 {  =}  O(z);\quad z \rightarrow 0.
\]

{\bigskip\bf\noindent Case III.~~}
As for the case $d=2\alpha$, the first pole in $A\cup B$ is $s=-1$, which is a simple pole.
As calculated before, the residue at this pole is vanishing, $h_{20}^* z\equiv 0$.
Hence, we need to consider the next pole at $s=-2$,
which is a pole of order $2$.
Use the asymptotic expansion \eqref{E:HZeroLog} to obtain that
\begin{align*}
\FoxH{2,1}{2,3}{z}
 {(1,1),\:(\beta,\beta)}{(d/2,\alpha/2),\:(1,1),\:(1,\alpha/2)} {  =}   O(z^2\log z) ;\qquad z \rightarrow 0.
\end{align*}
Finally, because $a^*>0$, by Theorem \ref{T:1.2(3)}, our H-function is a continuous function for $z>0$.
With this, we complete the proof of Lemma \ref{est.H}.
\end{proof}

\begin{lemma}\label{basic.ineq-1}
  {For all $\al\in (0, 2]$, $d\in\NN$ and $\kappa < \min\{2\alpha, d\}$,
one can find $\zeta<\min(d/\alpha,2)$ and a nonnegative constant $C$ 
(independent of $a$)  such that
\[
\int_{\RR^{ d}} |x -a|^{-\kappa}   \Theta(x) dx\le C<\infty\,\qquad\text{for all $a\in\R^d$,}
\]
where
\[
\Theta(x)=\frac{1}{ |x |^{\alpha+d}+ |x |^{d-\zeta \alpha}}\,.
\]
}
\end{lemma}
\begin{proof} We divide the integral domain into $\{|x|\leq 1\}$ and $\{|x| >1\}$.  Over  the domain
 $\{|x|\leq 1\}$, we have
 \begin{align*}
\int_{|x|\leq 1} &|x-a|^{-\kappa}\Theta(x)\ud x\leq  \int_{|x |\leq 1} |x-a|^{-\kappa}\frac{1}{  |x |^{d-\zeta \alpha}}\ud x \\
&=\int_{|x |\leq 1,\: |x |\leq |x-a|} |x-a|^{-\kappa}\frac{1}{  |x |^{d-\zeta \alpha}}\ud x
 +\int_{|x-a|<|x |\leq 1} |x-a|^{-\kappa}\frac{1}{  |x |^{d-\zeta \alpha}}\ud x \\
&\le \int_{|x |\leq 1,\: |x |\leq |x-a|} |x |^{-\kappa}\frac{1}{  |x |^{d-\zeta \alpha}}\ud x
 +\int_{|x-a|<|x |\leq 1} |x-a|^{-\kappa}\frac{1}{  |x-a |^{d-\zeta \alpha}}\ud x \\
&\leq  \:2 \int_{|z|\leq 1} \frac{1}{  |z|^{\kappa+d-\zeta \alpha}}\ud z
\leq  C.
\end{align*}
The last inequality is valid since  we can choose  $\zeta$ sufficiently close  to $\min(d/\alpha, 2)$ so that  $\kappa+d-\zeta \alpha<d$.
On the other hand, over  the domain
 $\{|x|> 1\}$, we have
\begin{align*}
\int_{|x |>1} &|x-a|^{-\kappa}\Theta(x  )\ud x  \leq  \int_{|x |> 1} |x-a|^{-\kappa}\frac{1}{ |x |^{\alpha+d}}\ud x \\
\leq &    \int_{ |x-a|\geq|x |> 1} |x-a|^{-\kappa}\frac{1}{ |x |^{\alpha+d}}\ud x  + \int_{ |x |>|x-a|> 1} |x-a|^{-\kappa}\frac{1}{ |x |^{\alpha+d}}\ud x \\
 & +  \int_{ |x |> 1\geq |x-a|} |x-a|^{-\kappa}\frac{1}{ |x |^{\alpha+d}}\ud x \\
 \leq&    \: 2 \int_{|z|> 1} \frac{1}{ |z|^{\alpha+d}}\ud z+\int_{ |z|\leq 1} |z|^{-\kappa}\ud z
 \leq  C.
\end{align*}
Note that the above   constant  $C$    {does not depend on $a$}.
\end{proof}

\begin{lemma}\label{basic.ineq}
Assume $\kappa < \min\{2\alpha, d\}$. Then for all $s, r>0$ and $  x_2\,,  y_2\in\RR^d$, we have that
\[
\int_{\RR^{2d}} \left|Y(s, x_1-x_2) Y(r, y_1-y_2)\right|\: |x_1-y_1|^{-\kappa} \ud x_1 \ud y_1
\leq C_{\alpha, \beta, d, \nu, \kappa} \:(s\: r)^\theta,
\]
where $C$   {does not depend on $  x_2$ and $y_2\in\RR^d$}, and
\[
\theta:=\beta-1-\frac{\beta}{2\alpha}\kappa.
\]
\end{lemma}
\begin{proof} We use the notation $\Theta(x)$ in the previous lemma.
By Lemma \ref{est.H} and the expression of $Y$ through Fox H-function \eqref{E:Yab},
  {we see that
for any $\zeta<\min( d/\al, 2)$,}
there  is a constant $C_{\alpha, \beta, d, \nu, \zeta}$
such that
\begin{eqnarray*}
\left|Y(t,x)\right|\
&\le& C_{\alpha, \beta, d, \nu, \zeta} |x|^{-d}t^{\be-1} \frac{|\frac{x}{t^{\be/\al}}|^{\al \zeta}}{
|\frac{x}{t^{\be/\al}}|^{\al \zeta+\al}+1}\\
&=& C_{\alpha, \beta, d, \nu, \zeta}  t^{\be-1-\frac{\be d}{\al}} \Theta\left(\frac{x}{t^{\be/\al}}\right)\,.
\end{eqnarray*}
Therefore,   {by Lemma \ref{basic.ineq}}, we have
\begin{eqnarray*}
&&\int_{\RR^{2d}} \left|Y(s, x_1-x_2) Y(r, y_1-y_2)\right|\:|x_1-y_1|^{-\kappa}\ud x_1 \ud y_1  \\
&&\qquad \le C_{\alpha, \beta, d, \nu, \zeta}   (sr) ^{\be-1-\frac{\be d}{\al}}
\int_{\RR^{2d}} \Theta\left(\frac{x_1-x_2}{s^{\be/\al}}\right)
 \Theta\left(\frac{y_1-y_2}{r^{\be/\al}}\right) |x_1-y_1|^{-\kappa}\ud x_1 \ud y_1  \\
&&\qquad \le C_{\alpha, \beta, d, \nu, \zeta}   r ^{\be-1-\frac{\be d}{\al}}s^{\be-1-\kappa \be/\al}
 \int_{\RR^d} \left(\int_{\RR^d} \left|z_1-\frac{y_1-x_2}{s^{\be/\al}}
 \right|^{-\kappa}\Theta(z_1  )\ud z_1\right) \Theta\left(\frac{y_1-y_2}{r^{\be/\al}}\right)\ud y_1\\
 &&\qquad \le C_{\alpha, \beta, d, \nu, \zeta}   r ^{\be-1-\frac{\be d}{\al}}s^{\be-1-\kappa \be/\al}
 \int_{\RR^d}   \Theta\left(\frac{y_1-y_2}{r^{\be/\al}}\right)\ud y_1\\
 &&\qquad \le C_{\alpha, \beta, d, \nu, \zeta}   r ^{\be-1 }s^{\be-1-\kappa \be/\al} \,.
\end{eqnarray*}
By symmetry, we also have
 \[
\iint_{\RR^{2d}} \left|Y(s, x_1-x_2) Y(r, y_1-y_2)\right|\: |x_1-y_1|^{-\kappa}\ud x_1 \ud y_1  \leq
C_{\alpha, \beta, d, \nu, \zeta}   s ^{\be-1 } r^{\be-1-\kappa \be/\al} \,.
\]
Now from the fact that   $c\le a$ and $c\le b$ implies $c\le \sqrt{ab}$,
the lemma follows.
\end{proof}

The following lemma is from \cite[Theorem 3.5]{BC3}.
\begin{lemma}
\label{stint}
  {Let $T_n$ be the simplex defined in \eqref{E:Tn}. Then for all $h>-1$, it holds that}
\[
 \int_{T_n(t)} [(t-s_n)(s_n-s_{n-1}) \ldots
(s_2-s_1)]^{h} \ud  s =\frac{\Gamma(1+h)^{n}}{\Gamma(n(1+h)+1)}
t^{n(1+h)}\,.
\]
\end{lemma}

\bigskip
\begin{proof}[Proof of Theorem \ref{beta.2}]
Following  the same notation and arguments as the proof of Theorem \ref{T:ExUni} until \eqref{E:fnNorm}, we have 
%
{  
\begin{multline*}
n! \| f_n(\cdot,\cdot,t,x)\|^2_{\mathcal{H}^{\otimes n}}\\
\leq C\frac{1}{n!}\int_{[0,t]^{2n}} \ud s\ud r\int_{\mathbb{R}^{2nd}}\ud y\ud z\: g_n(s, y, t, x)g_n(r, z, t, x)\prod_{i=1}^n\Lambda(y_i-z_i)\prod_{i=1}^n\gamma(s_i-r_i).
\end{multline*}
Furthermore,  by Cauchy-Schwarz inequality, we obtain 
\begin{align*}
\int_{\mathbb{R}^{2nd}}\ud y\ud z\: &g_n(s, y, t, x)g_n(r, z, t, x)\prod_{i=1}^n\Lambda(y_i-z_i)\\
\leq&\quad \left\{\int_{\mathbb{R}^{2nd}}\ud y\ud z\: g_n(s, y, t, x)g_n(s, z, t, x)\prod_{i=1}^n\Lambda(y_i-z_i)\right\}^{1/2}\\
&\times\left\{\int_{\mathbb{R}^{2nd}}\ud y\ud z\: g_n(r, y, t, x)g_n(r, z, t, x)\prod_{i=1}^n\Lambda(y_i-z_i)\right\}^{1/2}
\end{align*}
Applying Lemma \ref{basic.ineq} to the above two integrals, we have

\[
\int_{\mathbb{R}^{2nd}}\ud y\ud z\: g_n(s, y, t, x)g_n(r, z, t, x)\prod_{i=1}^n\Lambda(y_i-z_i)
\leq C^n_{\alpha, \beta, d, v, \kappa }(\phi(s)\phi(r))^\theta,
\]}
where
\[
 \phi(s) :=\prod_{i=1}^{n}(s_{\sigma(i+1)}- s_{\sigma(i)})  {\qquad\text{and}\qquad} \phi(r): = \prod_{i=1}^{n} (r_{\rho(i+1)}- r_{\rho(i)}),
\]
with
\[
0<s_{\sigma(1)}<s_{\sigma(2)}<
\ldots < s_{\sigma(n)} \quad \text{and} \quad 0<r_{\rho(1)}<r_{\rho(2)}< \ldots <
r_{\rho(n)} .
\]
Hence,
\begin{align*}
n! \| f_n(\cdot,\cdot,t,x)\|^2_{\mathcal{H}^{\otimes n}}
&\leq  \frac{C^n_{\alpha, \beta, d, v, \kappa }}{n!}\int_{[0,t]^{2n}} \prod_{i=1}^n \gamma(s_i-r_i) (\phi(s)\phi(r))^\theta    \ud s \ud r\\
&\leq  \frac{C^n_{\alpha, \beta, d, \nu, \kappa}}{n!} \frac{1}{2} \int_{[0,t]^{2n}} \prod_{i=1}^n \gamma(s_i-r_i) \left(\phi(s) ^{2\theta} + \phi(r)^{2\theta} \right)  \ud s \ud r\\
&=  \frac{C^n_{\alpha, \beta, d, \nu, \kappa}}{n!} \int_{[0,t]^{2n}} \prod_{i=1}^n \gamma(s_i-r_i) \phi(s) ^{2\theta} \ud s \ud r\\
&\leq \frac{C^n_{\alpha, \beta, d, \nu, \kappa} C_t^n}{n!} \int_{[0,t]^n} \phi(s)^{2\theta} \ud s\\
&=  C^n_{\alpha, \beta, d, \nu, \kappa} C_t^n  \int_{T_n(t)} \phi(s)^{2\theta} \ud s\\
&=  \frac{C^n_{\alpha, \beta, d, \nu, \kappa} C_t^n \Gamma(2\theta +1)^n t^{(2\theta+1)n}}{\Gamma((2\theta+1)n+1)}\,,
\end{align*}
where $C_t$ is defined in \eqref{E:Ct}.
Therefore,
\[
n! \| f_n(\cdot,\cdot,t,x)\|^2_{\mathcal{H}^{\otimes n}}\leq  \frac {C^n_{\alpha, \beta, d, \nu, \kappa} C_t^n}{\Gamma((2\theta+1)n+1)}\:,
\]
and
$\sum_{n\ge 0} n! \| f_n(\cdot,\cdot,t,x)\|^2_{\mathcal{H}^{\otimes n}}$ converges if $\theta>-1/2$.
Finally, the condition $\theta>-1/2$,
which is equivalent to $\kappa<2\alpha-\alpha/\beta$,
guarantees both condition $\theta> -1$  in Lemma \ref{stint} and the assumption $\kappa <2\alpha$ used in Lemma \ref{basic.ineq}.
This completes the proof of Theorem \ref{beta.2}.
\end{proof}

\appendix
\section{Fox H-function} \label{Sec:H}

\begin{definition}\label{D:H}
Let  $m, n, p, q$ be integers such that  $0\leq m\leq q,  0\leq n\leq p$.  Let
  $a_i,  b_i\in \mathbb{C}$ be complex numbers and
  let $\alpha_j, \beta_j$ be positive numbers,  $i=1, 2, \dots,  p$ and $j=1, 2, \dots, q$.
  Let the   set of poles of the gamma functions $\Gamma(b_j+\beta_js)$ doesn't intersect with that of the gamma functions $\Gamma(1-a_i-\alpha_is)$,
namely,
\begin{align}\label{E:poles}
\bigg\{b_{jl}=\frac{-b_j-l}{\beta_j},  l =0, 1, \cdots\bigg\}\bigcap \bigg\{a_{ik}=\frac{1-a_i+k}{\alpha_i},  k=0, 1, \cdots\bigg\}=\emptyset,
\end{align}
for all $i=1, 2, \dots,  p$ and $ j=1, 2,\dots,  q$.
Denote
\[
\mathcal{H}^{m,n}_{p,q}(s):=\frac{\prod_{j=1}^m \Gamma(b_j+\alpha_js)\prod_{i=1}^n\Gamma(1-a_i-\alpha_is)}{\prod_{i=n+1}^p\Gamma(a_j+\alpha_is)\prod_{j=m+1}^q \Gamma(1- b_j-\alpha_js)}\:.
\]
The {\em Fox H-function}
\begin{align}\label{E:FoxH}
H^{m,n}_{p,q}(z)\equiv H^{m,n}_{p,q}\bigg[z \bigg|\begin{array}{ccc}
(a_1, \alpha_1) & \cdots & (a_p, \alpha_p)\\
(b_1, \beta_1) & \cdots & (b_q, \beta_q)
\end{array} \bigg]
\end{align}
is defined by   the following integral
\begin{equation}\label{2.1}
H^{m,n}_{p,q}(z)=\frac{1}{2\pi i}\int_L \mathcal{H}^{m,n}_{p,q}(s) z^{-s} \ud s\,, \ \ z\in \mathbb{C}\,,
\end{equation}
where an empty product in \eqref{2.1}  means  $1$  and
$L$ in \eqref{2.1} is the infinite contour which separates all the points  $b_{jl}$ to the left and all the points
 $a_{ik}$ to the right of $L$.  Moreover, $L$  has one of the following forms:
\begin{enumerate}[(1)]
\item $L=L_{-\infty}$ is a left loop situated in a horizontal strip starting at point $-\infty+i\phi_1$ and terminating at point $-\infty+i\phi_2$ for some   $-\infty<\phi_1< \phi_2<\infty$;
\item $L=L_{+\infty}$ is a right loop situated in a horizontal strip starting at point $+\infty+i\phi_1$ and terminating at point $\infty+i\phi_2$ for some  $-\infty<\phi_1< \phi_2<\infty$;
\item $L=L_{i\gamma\infty}$ is a contour starting at point $\gamma-i\infty$ and terminating at point $\gamma+i\infty$ for some  $\gamma\in(-\infty,  \infty)$.
\end{enumerate}
\end{definition}

\begin{figure}[htbp]
\centering
\includegraphics[width=2in]{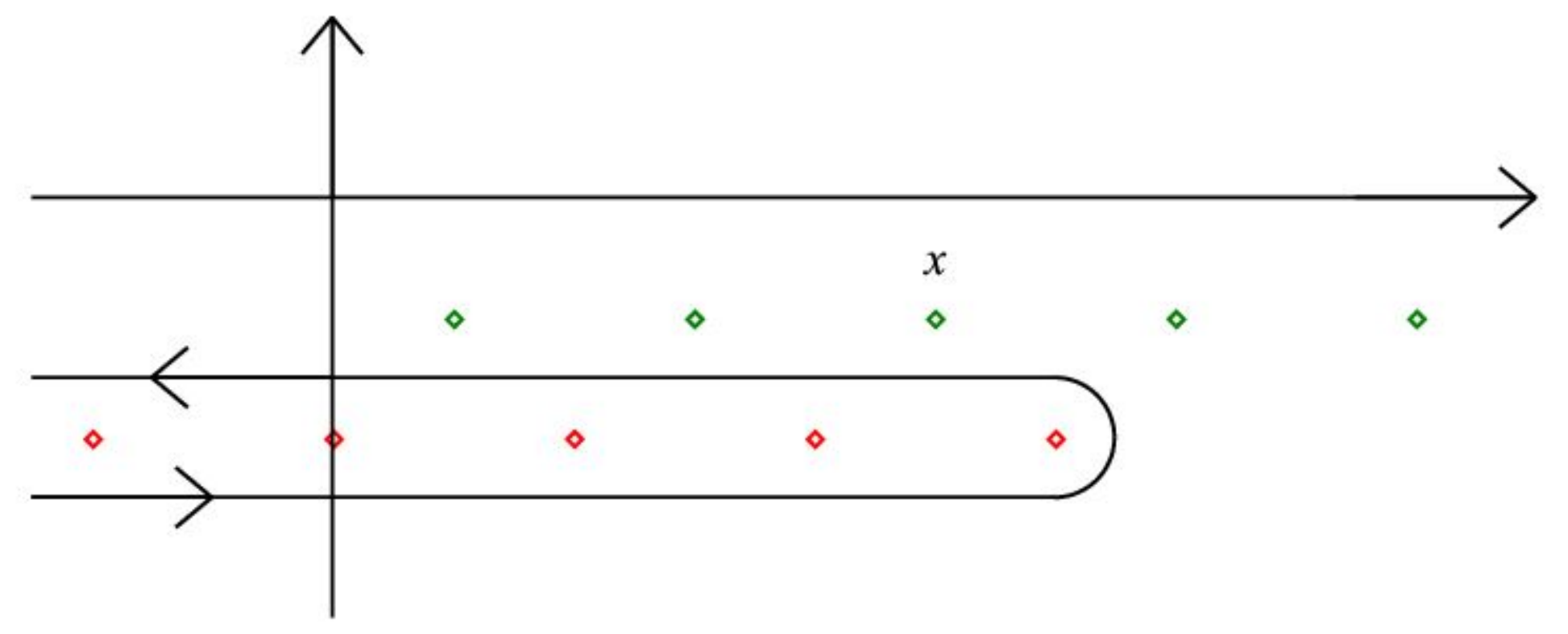}\quad
\includegraphics[width=2.in]{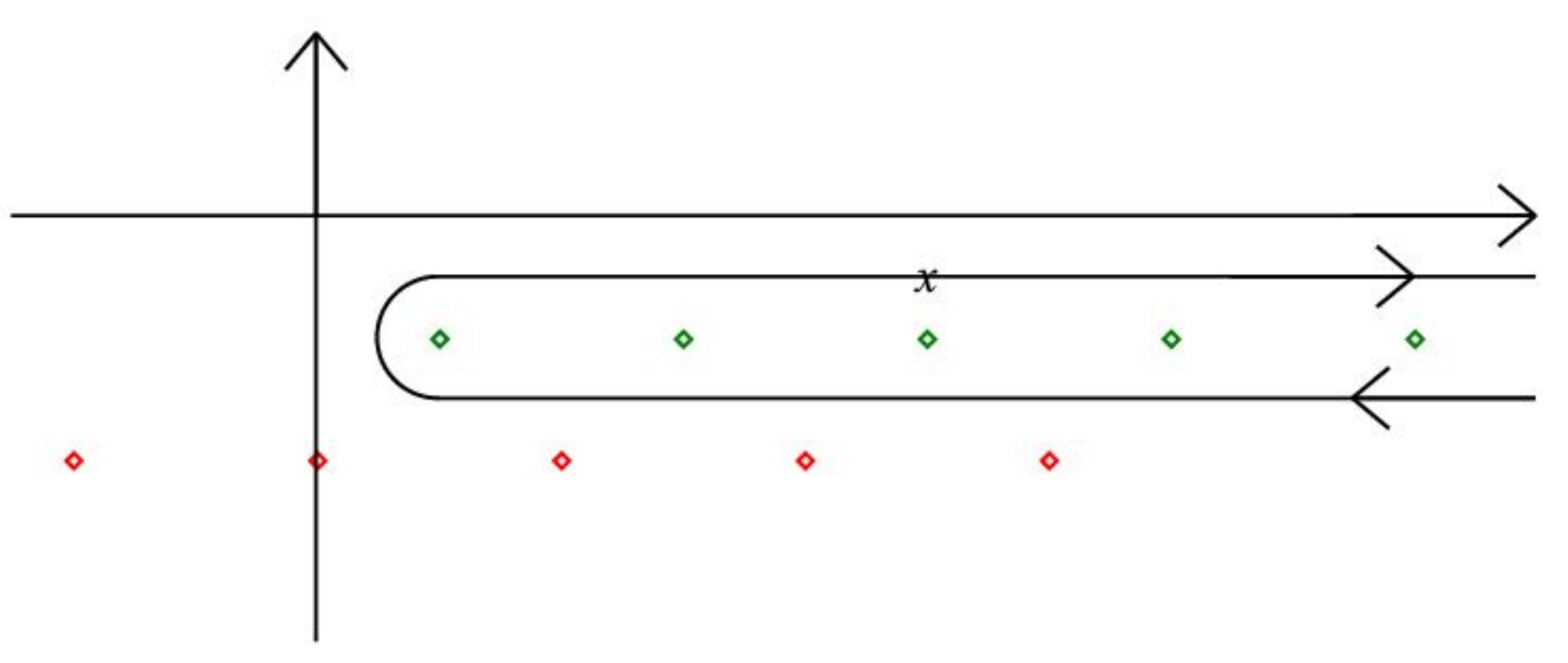}\quad
\includegraphics[width=2.in]{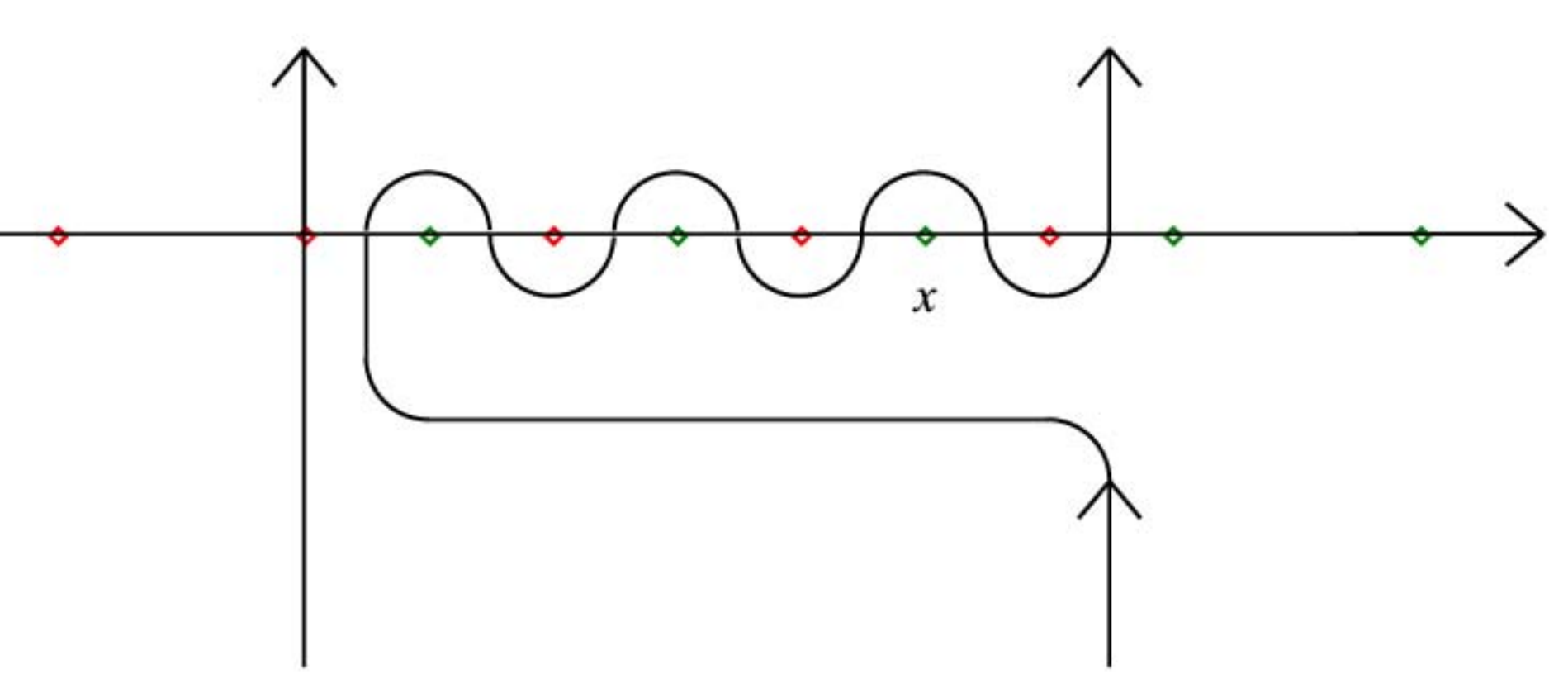}\\
\caption{Some illustrations of the path $L$.}
\label{F:L}
\end{figure}
See Figure \ref{F:L} for some illustrating paths.
According to \cite[Theorem 1.1]{KilbasSaigo04H}, the integral \eqref{2.1} exists, for example,
when
\begin{align}\label{E:Delta}
\Delta:=\sum_{j=1}^q\beta_j-\sum_{i=1}^p\alpha_i\geq0 \quad\text{and}\quad L=L_{-\infty},
\end{align}
or when
\begin{align}\label{E:a*}
a^*:=\sum_{i=1}^n \alpha_i -\sum_{i=n+1}^p\alpha_i+\sum_{j=1}^m\beta_j-\sum_{j=m+1}^{q}\beta_j\geq 0
\quad\text{and}\quad L=L_{i\gamma \infty}.
\end{align}

We call the poles $a_{ik}$ (see \eqref{E:poles}) of $\Gamma(1-a_i+\alpha_i s)$, $i=1,\dots, n$, {\it simple}, if 
\begin{equation}\label{pole.con1}
\forall i,j,\quad \alpha_j(1-a_i+k)\neq\alpha_i(1-a_j+l),\qquad i \neq j; \quad i, j=1, 2,\dots, n; \quad k,l=0, 1,\dots.
\end{equation}
Similarly, the poles $b_{jl}$ (see \eqref{E:poles}) of $\Gamma(b_j+\beta_j s)$, $j=1,\dots, m$, {\it simple}, if
\begin{equation}\label{pole.con2}
\forall i,j,\quad \beta_j(b_i+k)\neq\beta_i(b_j+l),\qquad i \neq j; \quad  i, j=1, 2,\dots, m; \quad k,l=0, 1,\dots.
\end{equation}

\begin{theorem}[Part (iii) of Theorem 1.2 in \cite{KilbasSaigo04H}]\label{T:1.2(3)}
\label{conti.H}
Let $L=L_{i\gamma \infty},  a^*>0$,  $z \neq 0,$ then $H^{m,n}_{p,q}(z)$ is analytic on $\{z: |arg z|<a^*\pi/2\}$.
\end{theorem}

\subsection{Some elementary properties}

\begin{property}[Property 2.2 of \cite{KilbasSaigo04H}]\label{Prop:Reduction}
If $n\ge 1$ and $q>m$, then
\[
\FoxH{m,n}{p,q}{z}{(a_i,\alpha_i)_{1,p}}{(b_j,\beta_j)_{1,q-1},\:(a_1,\alpha_1)}  =
\FoxH{m,n-1}{p-1,q-1}{z}{(a_i,\alpha_i)_{2,p}}{(b_j,\beta_j)_{1,q-1}} .
\]
And similarly, if $m\ge 1$ and $p>n$, then
\[
\FoxH{m,n}{p,q}{z}{(a_i,\alpha_i)_{1,p-1},\:(b_1,\beta_1)}{(b_j,\beta_j)_{1,q}}  =
\FoxH{m-1,n}{p-1,q-1}{z}{(a_i,\alpha_i)_{1,p-1}}{(b_j,\beta_j)_{2,q}} .
\]
\end{property}

\begin{property}[Property 2.3 of \cite{KilbasSaigo04H}]
\label{Prop:frac}
It holds that
\[
\FoxH{m,n}{p,q}{\frac{1}{z}}{(a_i,\alpha_i)_{1,p}}{(b_j,\beta_j)_{1,q}}
=
\FoxH{n,m}{q,p}{z}{(1-b_j,\beta_j)_{1,q}}{(1-a_i,\alpha_i)_{1,p}}.
\]
\end{property}

\begin{property}[Property 2.4 of \cite{KilbasSaigo04H}]
\label{Prop:Power}
For $k>0$, it holds that
\[
\FoxH{m,n}{p,q}{z}{(a_i,\alpha_i)_{1,p}}{(b_j,\beta_j)_{1,q}}  =
k \:
\FoxH{m,n}{p,q}{z^k}{(a_i,k\alpha_i)_{1,p}}{(b_j,k\beta_j)_{1,q}}.
\]
\end{property}

\subsection{Asymptotics at zero and infinity}

The following two theorems are some basic results on the asymptotic expansions of the Fox H-functions at zero and infinity.

\begin{theorem}\label{T:AsyInfty}
Suppose  $H^{m,n}_{p,q}(z)$ satisfies either $\Delta<0, a^*>0 $ or $\Delta \geq 0$.
If the poles of $\Gamma(1-a_i+\alpha_is)$ do not coincide, i.e., \eqref{pole.con1} holds,
then when $z\rightarrow \infty$, we have
\begin{equation} \label{0.infty}
 H^{m,n}_{p,q}(z)\sim\sum^n_{i=1}\sum^\infty_{k=0}h_{ik}z^{\frac{a_i-1-k}{\alpha_i}},
 \end{equation}
 where
\begin{equation}\label{infty.co}
h_{ik}=\frac{(-1)^k}{k!\alpha_i}\frac{\prod^m_{j=1}\Gamma\left(b_j+[1-a_i+k]\frac{\beta_j}{\alpha_i}\right) \prod^n_{j=1, i\neq j}\Gamma\left(1-a_j-[1-a_i+k]\frac{\alpha_j}{\alpha_i}\right) }{\prod^p_{j=n+1}\Gamma\left(a_j+[1-a_i+k]\frac{\alpha_j}{\alpha_i}\right) \prod^q_{j=m+1}\Gamma\left(1-b_j-[1-a_i+k]\frac{\beta_j}{\alpha_i}\right) }.
\end{equation}
\end{theorem}

\begin{proof}
See section 1.5 of  \cite{KilbasSaigo04H} . \eqref{0.infty} is asymptotic expansion 1.5.1 of \cite{KilbasSaigo04H}.
\end{proof}

\begin{theorem}\label{T:AsypZero}
Suppose  $H^{m,n}_{p,q}(z)$ satisfies either $\Delta<0$, $a^*>0$ or $\Delta \geq 0$. When $z\rightarrow0$, we have
the following two cases:
\begin{enumerate}[(1)]
\item If the poles $b_{jl}$ are simple (see \eqref{pole.con2}), then 
\begin{equation}\label{E:HZeroSimple}
H^{m,n}_{p,q}(z)\sim\sum^m_{j=1}\sum^\infty_{l=0}h^*_{jl}z^{\frac{b_j+l}{\beta_j}};
\end{equation}
\item If the poles $b_{jl}$ (see \eqref{E:poles}) of $\Gamma(b_j+\beta_js)$, $j=1,\dots, m$, coincide:
\[
\exists i,j, \qquad  \beta_j(b_i+k)=\beta_i(b_j+l),\qquad i \neq j; \quad i, j=1, 2,\dots, m; \quad k,l=0, 1, \dots,
\]
then 
\begin{equation}\label{E:HZeroLog}
H^{m,n}_{p,q}(z)\sim \sum\nolimits'_{j, l} h^*_{jl}z^{\frac{b_j+l}{\beta_j}}+\sum\nolimits''_{j, l}\sum^{N^*_{jl}-1}_{i=0}H^*_{jli}z^{\frac{b_j+l}{\beta_j}}[\log z]^i,
\end{equation}
\end{enumerate}
Here
\begin{equation}\label{E:h*jl}
h^*_{jl}=\frac{(-1)^l}{l!\beta_j}\frac{\prod^m_{i=1,i\neq j}\Gamma\left(b_i-[b_j+l]\frac{\beta_i}{\beta_j}\right) \prod^n_{i=1}\Gamma\left(1-a_i+[b_j+l]\frac{\alpha_i}{\beta_j}\right) }{\prod^p_{i=n+1}\Gamma\left(a_i-[b_j+l]\frac{\alpha_i}{\beta_j}\right) \prod^q_{i=m+1}\Gamma\left(1-b_i+[b_j+l]\frac{\beta_i}{\beta_j}\right) };
\end{equation}
$\sum_{j, l}  ' $ is summation over $j, l$ such that the $b_{jl}$
do not coincide; $\sum_{j, l} '' $
is the summation over $j, l$ such that $b_{jl}$ coincide with order $N^*_{jl}$;
\begin{equation}\label{com.co}
H^*_{jli}=\frac{1}{(N^*_{jl}-1)!}\sum_{n=i}^{N^*_{jl}-1}(-1)^i{N^*_{jl}-1 \choose n}{n \choose i}[\mathcal{H}^*_1(b_{jl})]^{(N^*_{jl}-1-n)}[\mathcal{H}^*_2(b_{jl})]^{(n-i)},
\end{equation}
where $N^*_{jl}-1-n$ and $n-i$ are orders of derivative;
\begin{equation}
\mathcal{H}^*_1(s)=(s-b)^{N^*}\prod_{j=j_1}^{j_{N^*}}\Gamma(b_j+\beta_js), \qquad \mathcal{H}^*_2(s)=(s-b)^{N^*}\prod_{j=j_1}^{j_{N^*}}\Gamma(b_j+\beta_js)\mathcal{H}^{m,n}_{p,q}(s),
\end{equation}
where $b$ is the pole with order $N^*.$
\end{theorem}
\begin{proof}
See section 1.8 of  \cite{KilbasSaigo04H} . \eqref{E:HZeroSimple} is asymptotic expansion 1.8.1 of \cite{KilbasSaigo04H}.
\eqref{E:HZeroLog} is asymptotic expansion 1.8.2  of \cite{KilbasSaigo04H}.
\end{proof}

\subsection{A convolution property}

The following theorem is a simplified version of Theorems 2.9 and 2.10 in \cite{KilbasEtc06}, which is
sufficient for our use. We need another parameter for the Fox H-function \eqref{E:FoxH}:
\begin{align}\label{E:mu}
\mu=\sum_{j=1}^q b_j -\sum_{i=1}^p a_i + \frac{p-q}{2}.
\end{align}
%

\begin{theorem}\label{T:HConvH}
Let $(a_1^*,\Delta_1,\mu_1)$ and $(a_2^*,\Delta_2,\mu_2)$ be the constants $(a^*,\Delta,\mu)$ defined in
\eqref{E:a*}, \eqref{E:Delta} and \eqref{E:mu} for the following two Fox H-functions:
\[
 \FoxH{m,n}{p,q}{x}{(a_i,\alpha_i)_{1,p}}{(b_j,\beta_j)_{1,q}}
 \quad\text{and}\quad
 \FoxH{M,N}{P,Q}{x}{(d_i,\delta_i)_{1,P}}{(c_j,\gamma_j)_{1,Q}},
\]
respectively. Denote
\begin{align*}
&A_1 = \min_{1\le i\le n} \frac{1-\Re(a_i)}{\alpha_i},\quad
B_1 = \min_{1\le j\le m} \frac{\Re(b_j)}{\beta_j},\\
&A_2 = \min_{1\le j\le M} \frac{\Re(c_j)}{\gamma_j},\quad
B_2 = \min_{1\le i\le N} \frac{1-\Re(d_i)}{\delta_i},
\end{align*}
with the convention that $\min(\phi)=+\infty$.
If either of the following four conditions holds
\begin{enumerate}[(1)]
 \item $a_1^*> 0$ and $ a_2^*\ge 0$ and $\Delta_2\ne 0$;
 \item $a_1^*\ge 0$ and $ a_2^*> 0$ and $\Delta_1\ne 0$;
 \item $a_1^*=\Delta_1=0$, $\Re(\mu_1)<-1$ and $a_2^*>0$;
 \item $a_2^*=\Delta_2=0$, $\Re(\mu_2)<-1$ and $a_1^*>0$;
 \item $a_1^*=\Delta_1=0$, $\Re(\mu_1)<-1$ and $a_2^*=\Delta_2=0$, $\Re(\mu_2)<-1$,
\end{enumerate}
and if
\begin{align}\label{E:HConvH}
A_1+B_1>0,\quad A_2+B_2>0,\quad A_1+A_2>0,\quad B_1+B_2>0,
\end{align}
then, for all $z>0$, $x\in\R$,
\begin{align*}
&\FoxH{m+M,n+N}{p+P,q+Q}{z x}{(a_i,\alpha_i)_{1,n},\: (d_i,\delta_i)_{1,P},\: (a_i,\alpha_i)_{n+1,p}}
 {(b_j,\beta_j)_{1,m},\: (c_j,\gamma_j)_{1,Q},\: (b_j,\beta_j)_{m+1,q}}
 \\
&\hspace{4em} =\int_0^\infty
 \FoxH{m,n}{p,q}{z t}{(a_i,\alpha_i)_{1,p}}{(b_j,\beta_j)_{1,q}}
 \FoxH{M,N}{P,Q}{\frac{x}{t}}{(d_i,\delta_i)_{1,P}}{(c_j,\gamma_j)_{1,Q}}\frac{\ud t}{t}.
\end{align*}
\end{theorem}
\begin{proof}
By Property \ref{Prop:frac},
\begin{align}\label{E_:HH}
 \FoxH{M,N}{P,Q}{\frac{x}{t}}{(d_i,\delta_i)_{1,P}}{(c_j,\gamma_j)_{1,Q}}
 =
 \FoxH{N,M}{Q,P}{\frac{t}{x}}{(1-c_j,\gamma_j)_{1,Q}}{(1-d_i,\delta_i)_{1,P}}.
\end{align}
If both $a_1^*>0$ and $a_2^*>0$, then one can apply Theorem 2.9 of \cite{KilbasEtc06} with $\eta=0$, $\sigma=1$, $w=1/x$, and with the following replacements:
$N\rightarrow M$, $M\rightarrow N$,
$P\rightarrow Q$, $Q\rightarrow P$,
$c_j\rightarrow 1-c_j$, $d_i\rightarrow 1-d_i$.
When $a_1^*>0$ but $a_2^*=0$ and $\Delta_2\ne 0$, because both $z$ and $x$ are real numbers,
Theorem 2.9 of \cite{KilbasSaigo04H} still holds (see the paragraph before Theorem 2.10 of \cite{KilbasSaigo04H} on p.59).
Therefore, we have proved the theorem under conditions (1) or (2).

If either of conditions (3)--(5) holds, we apply Theorem 2.10 of \cite{KilbasEtc06}
in the same way.
Note that the parameters $\mu$ for both Fox H-functions in \eqref{E_:HH} are equal.
\end{proof}

\subsection{Some integral transforms}
The following result on the Laplace transform of the Fox H-function is used in this paper.
Denote
\[
(\bL f)(t):=\int_0^\infty f(x)e^{-t x}\ud x, \quad t\in\bC.
\]

\begin{theorem}[Part of Corollary 2.3.1 of \cite{KilbasSaigo04H}]\label{T:Laplace}
Suppose that $a^*>0$. Assume that $w\in\mathbb{C}$, $a>0$ and $\sigma>0$ are such that
\begin{align}\label{E:sigmabBetaw}
\sigma \min_{1\le j\le m}\frac{\Re(b_j)}{\beta_j} +\Re(w)>-1.
\end{align}
Then
\[
\left[\bL x^w\FoxH{m,n}{p,q}{a x^\sigma}{(a_i,\alpha_i)_{1,p}}{(b_j,\beta_j)_{1,q}}\right](t)
=
\frac{1}{t^{w+1}}
\FoxH{m,n+1}{p+1,q}{\frac{a}{t^\sigma}}{(-w,\sigma),\:(a_i,\alpha_i)_{1,p}}{(b_j,\beta_j)_{1,q}}.
\]
\end{theorem}

We also need the following convolution result, which is related to the Hankel transform of the Fox H-function.

\begin{theorem}(Corollary 2.5.1 of \cite{KilbasSaigo04H})\label{T:Hankel}
 Let $a^*>0$ or $a^*=\Delta=0$ and $\Re(\mu)<-1$.
 Let $\eta,w\in\bC$, $\tau>0$ and $\sigma>0$ be such that
 \begin{gather*}
  \sigma \Re(\eta)+\Re(w)+\tau\min_{1\le j\le m}\frac{\Re(b_j)}{\beta_j} >-1,\\
 \tau \min_{1\le i\le n}\frac{1-\Re(a_i)}{\alpha_i}>\Re(w)-\frac{\sigma}{2}+1,
 \end{gather*}
 and
 \[
 \Re(\eta)>-1/2.
 \]
Then for all $a>0$ and $b>0$,
\begin{align*}
\int_0^\infty & (xt)^w J_\eta\left(a (xt)^\sigma\right)
\FoxH{m,n}{p,q}{bt^\tau}{(a_i,\alpha_i)_{1,p}}{(b_j,\beta_j)_{1,q}}\ud t\\
&=\frac{1}{2\sigma x}\left(\frac{2}{a}\right)^{\frac{w+1}{\sigma}}
\FoxH{m,n+1}{p+2,q}{\frac{b}{x^\tau}\left(\frac{2}{a}\right)^{\tau/\sigma}}{\left(1-\frac{w+1}{2\sigma}-\frac{\eta}{2},\frac{\tau}{2\sigma}\right),(a_i,\alpha_i)_{1,p},\left(1-\frac{w+1}{2\sigma}+\frac{\eta}{2},\frac{\tau}{2\sigma}\right)}{(b_j,\beta_j)_{1,q}},\quad x>0.
\end{align*}

\end{theorem}

\subsection{Riemann-Liouville fractional derivatives}
Recall that the Riemann-Liouville fractional derivative of order $\alpha\in(0,1)$ is defined in \eqref{E:RL-D}.
The following results are used several times in this paper, which is a special case of part (i) of Theorem 2.8 of \cite{KilbasSaigo04H}.
\begin{theorem}\label{T:RL-D}
Let $\alpha\in (0,1)$, $w\in\bC$ and $\sigma>0$.
If $a^*>0$ and \eqref{E:sigmabBetaw} holds, then
\[
\left[D_{0+}^\alpha t^w \FoxH{m,n}{p,q}{t^\sigma}{(a_i,\alpha_i)_{1,p}}{(b_j,\beta_j)_{1,q}}\right](x)
=x^{w-\alpha}
\FoxH{m,n+1}{p+1,q+1}{x^\sigma}{(-w,\sigma),(a_i,\alpha_i)_{1,p}}{(b_j,\beta_j)_{1,q},(-w+\alpha,\sigma)}.
\]
\end{theorem}

\subsection{Some special cases}

The Mittag-Leffler function is a special case of the Fox H-function (see (2.9.27) of \cite{KilbasSaigo04H}):
\begin{align}\label{E:ML-H}
 E_{\rho,\mu} (x)= \FoxH{1,1}{1,2}{-x}{(0,1)}{(0,1),\:(1-\mu,\rho)},\qquad \rho>0, \: \mu\in\mathbb{C}.
\end{align}
The $cos(\cdot)$ function can be represented as (see \cite[(2.9.8)]{KilbasSaigo04H}):
\begin{align}\label{E:cos}
 \cos(z) &= \sqrt{\pi}\: \FoxH{1,0}{0,2}{\frac{z^2}{4}}{\hline}{(0,1),\: (1/2,1)}.
\end{align}
The Bessel function of the first kind $J_\eta(z)$ is equal to (see (2.9.18) of \cite{KilbasSaigo04H})
\begin{align}\label{E:BesselJ}
J_\eta(z) = \left(\frac{2}{z}\right)^a\:
\FoxH{1,0}{0,2}{\frac{z^2}{4}}{\hline}{\left(\frac{a+\eta}{2},1\right),\:\left(\frac{a-\eta}{2},1\right)}.
\end{align}
Another special case is (see (2.9.4) of \cite{KilbasSaigo04H})
\begin{align}
\label{E:zExp}
\FoxH{1,0}{0,1}{z}{\hline}{(b,\beta)} & =\frac{1}{\beta}z^{b/\beta}\exp\left(-z^{1/\beta}\right).
\end{align}

\section*{Acknowledgements}
Le Chen thanks Francesco Mainardi for some useful discussion and
thanks Anatoly N. Kochubei for pointing out Pskhu's paper \cite{Pskhu09}.


\begin{thebibliography}{99}

\bibitem{BC1}R.~Balan, D.~Conus.
\newblock Intermittency for the wave and heat equations with fractional noise in time.
\newblock {\it Ann. Probab. } to appear (2015).

\bibitem{BC2}R. Balan, D. Conus.
\newblock A note on intermittency for the fractional heat equation.
\newblock {\it Statist. Probab. Lett.} 95 (2014), 6--14.

\bibitem{BC3}R. Balan, C. Tudor.
\newblock Stochastic heat equation with multiplicative fractional-colored noise.
\newblock {\it J. Theor. Probab.} 23 (2010), 834--870.

\bibitem{Chen14Time}
L.~Chen.
\newblock Nonlinear stochastic time-fractional diffusion equations on $\RR$:
moments, H\"older regularity and intermittency.
\newblock {\it Preprint arXiv:1410.1911}, (2014).

\bibitem{ChenDalang14Wave}
L.~Chen, R.~C. Dalang.
\newblock Moment bounds and asymptotics for the stochastic wave
  equation.
\newblock {\em Stochastic Process. Appl.} 125 (2015), no. 4, 1605--1628.

\bibitem{ChenDalang14FracHeat}
L.~Chen, R.~C. Dalang.
\newblock Moments, intermittency, and growth indices for the nonlinear fractional stochastic heat equation.
\newblock {\em Stoch. Partial Differ. Equ. Anal. Comput.}, to appear, 2015.

\bibitem{ChenDalang13Heat}
L.~Chen, R.~C. Dalang.
\newblock Moments and growth indices for nonlinear stochastic heat equation
  with rough initial conditions.
\newblock {\em Ann. Probab.}, (to appear), 2014.

\bibitem{CHD15}
L.~Chen, Y.~Hu, D.~Nualart.
\newblock Nonlinear stochastic time-fractional slow and fast diffusion equations.
\newblock {\em Preprint}, 2015.



\bibitem{ChenKimKim15}
Z.~Q.~Chen, K.~H.~Kim, P.~Kim.
\newblock Fractional time stochastic partial differential equations.
\newblock {\it Stochastic Process. Appl.} 125 (2015), no. 4, 1470--1499.

\bibitem{Dalang99Extending}
R.~C. Dalang.
\newblock Extending the martingale measure stochastic integral with
  applications to spatially homogeneous s.p.d.e.'s.
\newblock {\em Electron. J. Probab.}, 4 (1999),  no.   6.


\bibitem{Die04}
K.~Diethelm.
\newblock {\em The analysis of fractional differential equations.
An application-oriented exposition using differential operators of Caputo type.}
\newblock Lecture Notes in Mathematics, 2004. Springer-Verlag, Berlin, 2010.

\bibitem{EK}
S.~Eidelman, A.~Kochubei.
\newblock Cauchy problem for fractional diffusion equations.
\newblock {\it J. Differential Equations} 199 (2004), no. 2, 211--255.



\bibitem{FK08Int}
M.~Foondun and D.~Khoshnevisan.
\newblock Intermittence and nonlinear parabolic stochastic partial differential
  equations.
\newblock {\em Electron. J. Probab.}  14 (2009), 548--568.

\bibitem{FN15}
M.~Foondun, E.~Nane.
\newblock Asymptotic properties of some space-time fractional stochastic equations.
\newblock {\em Preprint at arXiv:1505.04615}, (2015).


\bibitem{Hu}
Y.~Hu.
\newblock Heat equations with fractional white noise potentials.
\newblock {\it Appl. Math. Optim.} 43 (2001), no. 3, 221--243.

\bibitem{Hu1}
Y.~Hu.
\newblock Chaos expansion of heat equations with white noise potentials.
\newblock {\it Potential Anal}. 16 (2002), no. 1, 45--66.

\bibitem{HuHu15}
G.~Hu and Y.~Hu.
\newblock Fractional diffusion in Gaussian noisy environment.
\newblock {\em Mathematics}   3 (2015), 131-152; doi:10.3390/math3020131.

\bibitem{HHNT}
Y.~Hu, J.~Huang, D.~Nualart, S.~Tindel.
\newblock Stochastic heat equations with general multiplicative Gaussian noises:
H\"older continuity and intermittency.
\newblock {\it Electron. J. Probab.}  20, (2015), no. 55, 1-50.

\bibitem{HHLNT}
Y.~Hu, J.~Huang, K.~Le, D.~Nualart, S.~Tindel.
\newblock    Stochastic heat equation
with rough dependence in space.
\newblock {\it Preprint arXiv:1505.04924,} 2015.

\bibitem{huyan} Y.~Hu, J. ~Yan.  \newblock
Wick calculus for nonlinear Gaussian functionals.
\newblock {\it Acta Math. Appl.
Sin. Engl. Ser.} 25 (2009),  399-414.


\bibitem{KX}
D.~Khoshnevisan, Y.~Xiao.
\newblock Harmonic analysis of additive L\'evy processes.
\newblock {\it Probab. Theory Related Fields} 145 (2009), no. 3-4, 459--515.


\bibitem{KilbasSaigo04H}
A.~A.~Kilbas. and M.~Saigo.
\newblock {\em H-transforms: theory and applications.}
\newblock Analytical Methods and Special Functions, 9. Chapman \& Hall/CRC, Boca Raton, FL, 2004.

\bibitem{KilbasEtc06}
A.~A.~Kilbas, H.~M.~Srivastava, J.~J.~Trujillo.
\newblock {\em Theory and applications of fractional differential equations.}
\newblock North-Holland Mathematics Studies, 204. Elsevier Science B.V., Amsterdam, 2006.

\bibitem{Koc}
A.~Kochube\u {\i}.
\newblock Diffusion of fractional order.
\newblock 
{\it Differential Equations }26 (1990), no. 4, 485--492.

\bibitem{MLP}
F.~Mainardi,  Y.~Luchko, G.~Pagnini.
\newblock The fundamental solution of the space-time fractional diffusion equation.
\newblock {\it  Fract. Calc. Appl. Anal.} 4 (2001), no. 2, 153--192.

\bibitem{MijenaNana14Int}
J.~Mijena, E.~Nane.
\newblock Intermittence and time fractional stochastic partial differential equations.
\newblock {\it Preprint	arXiv:1409.7468} (2014).

\bibitem{MijenaNana15ST}
J.~Mijena, E.~Nane.
\newblock Space-time fractional stochastic partial differential equations.
\newblock {\it Stochastic Process. Appl.} to appear (2015).

\bibitem{Nua}
D.~Nualart.
\newblock {\em The Malliavin calculus and related topics.} Second edition.
\newblock Probability and its Applications (New York). Springer-Verlag, Berlin.

\bibitem{NIST2010}
F.~W.~J. Olver, D.~W. Lozier, R.~F. Boisvert, and C.~W. Clark, editors.
\newblock {\em N{IST} handbook of mathematical functions}.
\newblock U.S. Department of Commerce National Institute of Standards and Technology, Washington, DC, 2010.


\bibitem{Podlubny99FDE}
I.~Podlubny.
\newblock {\em Fractional differential equations},
 volume 198 of {\em  Mathematics in Science and Engineering}.
\newblock Academic Press Inc., San Diego, CA, 1999.

\bibitem{Pskhu09}
A.~V.~Pskhu.
\newblock The fundamental solution of a diffusion-wave equation of fractional order.
\newblock 
{\it Izv. Math.} 73 (2009), no. 2, 351--392.


\bibitem{Sch96CM}
W.~Schneider.
\newblock Completely monotone generalized Mittag-Leffler functions.
\newblock {\it Exposition. Math.} 14 (1996), no. 1, 3--16.

\bibitem{SchWyss89}
W.~Schneider, W.~Wyss.
\newblock Fractional diffusion and wave equations.
\newblock {\it J. Math. Phys.} 30 (1989), no. 1, 134--144.

\bibitem{Song15}
J.~Song.
\newblock On a class of stochastic partial differential equations.
\newblock {\it Preprint arXiv:1503.06525v2}, (2015).

\bibitem{Ste}
E.~Stein.
\newblock {\it Singular integrals and differentiability properties of functions.}
\newblock Princeton Mathematical Series, No. 30  Princeton University Press, Princeton, N. J., 1970.


\bibitem{SteinWeiss71}
E.~Stein, G.~Weiss.
\newblock {\it Introduction to Fourier analysis on Euclidean spaces.}
\newblock Princeton Mathematical Series, No. 32. Princeton University Press, Princeton, N. J., 1971.

\bibitem{UchaikinZolotarev99}
V.~V.~Uchaikin, V.~M.~Zolotarev.
\newblock {\it Chance and stability. Stable distributions and their applications.}
\newblock Modern Probability and Statistics. VSP, Utrecht, 1999.
%


\bibitem{Widder41LT}
D.~V. Widder.
\newblock {\em The {L}aplace transform}.
\newblock Princeton Mathematical Series, v. 6. Princeton University Press, Princeton, N. J., 1941.

\end{thebibliography}
\end{document}